%% file: main.tex
\documentclass[11pt]{article}
\usepackage[margin=1in]{geometry}
\usepackage{amsmath}
\usepackage{amsfonts}
\usepackage{color}
\usepackage{latexsym}
\usepackage[colorlinks]{hyperref}
\usepackage{epsfig}
\usepackage{float}
\usepackage{array}
\usepackage{amssymb,amsthm}
\usepackage{algorithm,algorithmic}

\usepackage{url}

\include{def}

\begin{document}
\title{Primal-dual proximal bundle and conditional gradient methods for convex problems
}
\date{November 30, 2024 (revisions: December 23, 2024; June 1, 2025; September 24, 2025)}

\author{
		Jiaming Liang \thanks{Goergen Institute for Data Science and Artificial Intelligence (GIDS-AI) and Department of Computer Science, University of Rochester, Rochester, NY 14620 (email: {\tt jiaming.liang@rochester.edu}). This work was partially supported by GIDS-AI seed funding and AFOSR grant FA9550-25-1-0182.
		}
}
	\maketitle
\maketitle

\begin{abstract}
This paper studies the primal-dual convergence and iteration-complexity of proximal bundle methods for solving nonsmooth problems with convex structures. More specifically, we develop a family of primal-dual proximal bundle methods for solving convex nonsmooth composite optimization problems and establish the iteration-complexity in terms of a primal-dual gap. We also propose a class of proximal bundle methods for solving convex-concave nonsmooth composite saddle-point problems and establish the iteration-complexity to find an approximate saddle-point. This paper places special emphasis on the primal-dual perspective of the proximal bundle method. In particular, we discover an interesting duality between the conditional gradient method and the cutting-plane scheme used within the proximal bundle method. Leveraging this duality, we further develop novel variants of both the conditional gradient method and the cutting-plane scheme. Additionally, we report numerical experiments to demonstrate the effectiveness and efficiency of the proposed proximal bundle methods in comparison with the subgradient method for solving a regularized matrix game.\\
		
		\\
		
	\end{abstract}

\section{Introduction}

This paper considers two nonsmooth problems with convex structures: 1) the convex nonsmooth composite optimization (CNCO) problem
\begin{equation}\label{eq:problem}
    \phi_*:=\min \{\phi(x):=f(x)+h(x): x \in \R^n\},
\end{equation}
where $f,h:\R^n\to \R \cup \{+\infty\}$ are proper lower semi-continuous convex functions such that $\dom h \subset \dom f$;
and 2) the convex-concave nonsmooth composite saddle-point problem (SPP)
\begin{equation}\label{eq:spp}
    \min_{x\in \R^n} \max_{y \in \R^m} \left\{\phi(x,y) := f(x,y) + h_1(x) - h_2(y)\right\},
\end{equation}
where $f(x,y)$ is convex in $x$ and concave in $y$, and $h_1:\R^n\to \R \cup \{+\infty\}$ and $h_2:\R^m\to \R \cup \{+\infty\}$ are proper lower semi-continuous convex functions such that $\dom h_1 \times \dom h_2 \subset \dom f$. The main goal of this paper is to study the primal-dual convergence and iteration-complexity of proximal bundle (PB) methods for solving CNCO and SPP.

Classical PB methods, first proposed in \cite{lemarechal1975extension,wolfe1975method} and further developed in \cite{lemarechal1978nonsmooth,mifflin1982modification}, are known to be efficient algorithms for solving CNCO problems.
At the core of classical PB methods is the introduction of a proximal regularization term to the standard cutting-plane method (or Kelly's method) and a sufficient descent test.
Those methods update the prox center (i.e., perform a serious step) if there is a sufficient descent in the function value; otherwise, they keep the prox center and refine the cutting-plane model (i.e., perform a null step).
Various bundle management policies (i.e., update schemes on cutting-plane models) have been discussed in
\cite{du2017rate,frangioni2002generalized,kiwiel2000efficiency,de2014convex,ruszczynski2011nonlinear,van2017probabilistic}.
The textbooks \cite{ruszczynski2011nonlinear,urruty1996convex} provide a comprehensive discussion of the convergence analysis of classical PB methods for CNCO problems.
Iteration-complexity bounds have been established in \cite{astorino2013nonmonotone,diaz2023optimal,du2017rate,kiwiel2000efficiency} for classical PB methods for solving CNCO problems \eqref{eq:problem} with $h\equiv 0$ or being the indicator function of a nonempty closed convex set.
Notably, the first complexity of classical PB methods is given by \cite{kiwiel2000efficiency} as ${\cal O} (\bar \varepsilon^{-3})$ to find a $\bar \varepsilon$-solution of \eqref{eq:problem} (i.e., a point $\bar x\in \dom h$ satisfying $\phi(\bar x)-\phi_*\le \bar \varepsilon$).

Since the lower complexity bound of CNCO is $\Omega(\bar \varepsilon^{-2})$ (see for example Subsection~7.1 of \cite{liang2021proximal}), it is clear that the bound ${\cal O} (\bar \varepsilon^{-3})$ given by \cite{kiwiel2000efficiency} is not optimal. Recent papers \cite{liang2021proximal,liang2024unified} establish the optimal complexity bound ${\cal O} (\bar \varepsilon^{-2})$ for a large range of prox stepsizes by developing modern PB methods, where the sufficient descent test in classical PB methods is replaced by a different serious/null decision condition motivated by the proximal point method (PPM) (see Subsection~3.1 of \cite{liang2021proximal} and Subsection~3.2 of \cite{liang2024unified}).
Moreover, \cite{liang2024unified} studies the cutting-plane (i.e., multi-cuts) model, the cut-aggregation (i.e., two-cuts) model, and a newly proposed one-cut model under a generic bundle update scheme, and provides a unified analysis for all models encompassed within this general update scheme.

This paper investigates the modern PB methods for solving CNCO problems from the primal-dual perspective. More specifically, it shows that a cycle (consecutive null steps between two serious steps) of the methods indeed finds an approximate primal-dual solution to a proximal subproblem, and further establishes the iteration-complexity of the modern PB methods in terms of a primal-dual gap of \eqref{eq:problem}, which is a stronger convergence guarantee than the $\bar \varepsilon$-solution considered in \cite{liang2021proximal,liang2024unified}. Furthermore, the paper reveals an interesting dual relationship between the conditional gradient (CG) method and
the cutting-plane scheme for solving proximal subproblems within PB. Extending upon this duality, the paper also develops novel variants of both CG and the cutting-plane scheme, drawing inspiration from both perspectives of the dual relationship.  

An independent study conducted concurrently by \cite{fersztand2024proximal} examines the same duality under a more specialized assumption that $f$ is piece-wise linear and $h$ is smooth. Building upon the duality and using the convergence analysis of CG, \cite{fersztand2024proximal} is able to improve the general complexity bound ${\cal O}(\bar \varepsilon^{-2})$ to ${\cal O}(\bar \varepsilon^{-4/5})$ in this context.
The duality relationship between the subgradient method/mirror descent and CG is first studied in \cite{bach2015duality}. Related works \cite{bach2013learning,bertsekas2011unifying,mazanti2024nonsmooth,zhou2018limited} investigate the duality between Kelly's method/simplicial method and CG across various settings, and also examine the primal and dual simplicial methods.

The second half of the paper is devoted to developing modern PB methods for solving convex-concave nonsmooth composite SPP. While subgradient-type methods have been extensively studied for solving such SPP, for example, \cite{nedic2009subgradient,golshtein1974generalized,maistroskii1977gradient,nesterov2009primal,uzawa1958iterative,zabotin1988subgradient}, PB methods, which generalize subgradient methods by better using the history of subgradients, have received less attention in this context. Inspired by the PPM interpretation of modern PB methods, this paper proposes a generic inexact proximal point framework (IPPF) to solve SPP \eqref{eq:spp}, comprising both a composite subgradient method and a PB method as special instances. The paper finally establishes the iteration-complexity bounds for both methods to find an approximate saddle-point of \eqref{eq:spp}.

{\bf Organization of the paper.}
Subsection~\ref{subsec:DefNot} presents basic definitions and notation used throughout the paper.
Section~\ref{sec:PDPB} describes the primal-dual proximal bundle (PDPB) method and the assumptions on CNCO, and establishes the iteration-complexity of PDPB in terms of a primal-dual gap.
In addition, Subsection~\ref{subsec:PDCP} presents the key subroutine, namely a primal-dual cutting-plane (PDCP) scheme,  used within PDPB for solving a proximal subproblem and provides the primal-dual convergence analysis of PDCP.
Section~\ref{sec:FW} explores the duality between PDCP and CG by demonstrating that PDCP applied to the proximal subproblem produces the same iterates as CG applied to the dual problem.
Subsection~\ref{subsec:alt-pf} presents an alternative primal-dual convergence analysis of PDCP using CG duality.
Moreover, inspired by the duality, Subsections~\ref{subsec:new GBM} and~\ref{subsec: new CG} develop novel PDCP and CG variants, respectively.
Section~\ref{sec:PB-SPP} extends PB to solving the convex-concave nonsmooth composite SPP.
More specifically, Subsection~\ref{subsec:IPPF} introduces the IPPF for SPP, Subsection~\ref{subsec:PB-SPP} describes the PB method for SPP (PB-SPP) and establishes its iteration-complexity to find an approximate saddle-point, and Subsection~\ref{subsec:opt} derives a tighter (and optimal) complexity bound compared with the one established in Subsection~\ref{subsec:PB-SPP}.
Section~\ref{sec:numerics} presents a comparison of the subgradient method with several variants of PB-SPP for solving a regularized matrix game.
Section~\ref{sec:conclusion} presents some concluding remarks and possible extensions.
Appendix~\ref{sec:technical} provides a few useful technical results and deferred proofs.
Appendices~\ref{sec:PDS} and~\ref{sec:CS-SPP} are devoted to the complexity analyses of subgradient methods for solving CNCO \eqref{eq:problem} and SPP \eqref{eq:spp}, respectively.
Appendix~\ref{sec:detail} provides further implementation details for the numerical experiments reported in Section~\ref{sec:numerics}.

\subsection{Basic definitions and notation} \label{subsec:DefNot}
    
    Let $\R$ denote the set of real numbers. Let $\R_{++}$ denote the set of positive real numbers. Let $\R^n$ denote the standard $n$-dimensional Euclidean space equipped with inner product and norm denoted by $\left\langle \cdot,\cdot\right\rangle $ and $\|\cdot\|$, respectively. 
    
    For given $f: \R^n\rightarrow (-\infty,+\infty]$, let $\dom f:=\{x \in \R^n: f (x) <\infty\}$ denote the effective domain of $f$. We say $f$ is proper if $\dom f \ne \emptyset$. A proper function $f: \R^n\rightarrow (-\infty,+\infty]$ is $\mu$-strongly convex for some $\mu > 0$ if for every $x, y \in \dom f$ and $t \in [0,1]$
    \[
    f(t x+(1-t) y)\le t f(x)+(1-t)f(y) - \frac{t(1-t) \mu}{2}\|x-y\|^2.
    \]
Let $\bConv{n}$ denote the set of all proper lower-semicontinuous convex functions.
	For $\varepsilon \ge 0$, the \emph{$\varepsilon$-subdifferential} of $ f $ at $x \in \dom f$ is denoted by
	\begin{equation}\label{def:subdiff}
	    \partial_\varepsilon f (x):=\left\{ s \in\R^n: f(y)\geq f(x)+\left\langle s,y-x\right\rangle -\varepsilon, \forall y\in\R^n\right\}.
	\end{equation}
	We denote the subdifferential of $f$ at $x \in \dom f$ by $\partial f (x)$, which is the set  $\partial_0 f(x)$ by definition.
    For a given subgradient
$f'(x) \in \partial f(x)$, we denote the linearization of convex function $f$ at $x$ by $\ell_f(\cdot;x)$, which is defined as
\begin{equation}\label{linfdef}
\ell_f(\cdot;x):=f(x)+\langle f'(x),\cdot-x\rangle.
\end{equation}
The infimum convolution of proper functions $f_1, f_2: \R^n\rightarrow (-\infty,+\infty]$ is given by
\begin{equation}\label{def:infconv}
    (f_1 \square f_2)(x)
= \min_{u\in \R^n}\{f_1(u) + f_2(x-u)\}.
\end{equation}

\section{Primal-dual proximal bundle method for CNCO} \label{sec:PDPB}

In this section, we consider the CNCO problem \eqref{eq:problem}.
More specifically, we assume the following conditions hold:
\begin{itemize}
    \item[(A1)] a subgradient oracle, i.e., a function $f': \dom h \to \R^n$ satisfying $f^{\prime}(x) \in \partial f(x)$ for every $x \in \dom h$, is available;
    \item[(A2)] $\|f'(x)\|\le M$ for every $x\in \dom h$ and some $M>0$;
    \item[(A3)] the set of optimal solutions $X_*$ of problem \eqref{eq:problem} is nonempty.
\end{itemize}
Define the linearization of $f$ at $x\in \dom h$, $\ell_f:\dom h \to \R$ as
\[
\ell_f(\cdot;x):=f(x)+\inner{f'(x)}{\cdot-x}.
\]
Clearly, it follows from (A2) that for every $x, y \in \dom h$,
\begin{equation}\label{ineq:f}
  f(x)-\ell_f(x ; y) \le 2 M\|x-y\|.  
\end{equation}
For a given initial point $\hat x_0 \in \dom h$, we denote its distance to $X_*$ as
\begin{equation}\label{eq:d0}
    d_0:=\left\|x_0^*-\hat x_0\right\|, \quad x_0^*:=\underset{x_* \in X_*}\argmin\{\left\| x_*-\hat x_0\right\|\}.
\end{equation}

The primal-dual subgradient method denoted by PDS$(\hat x_0,\lam)$, where $\hat x_0\in \dom h$ is the initial point and $\lam>0$ is the prox stepsize, recursively computes
\begin{equation}\label{eq:pds}
    s_k = f'(x_{k-1}) \in \partial f(x_{k-1}), \quad \hat x_k = \underset{u\in\R^n}{\argmin}\left\{\ell_f(u;\hat x_{k-1}) + h(u) + \frac{1}{2\lam}\|u-\hat x_{k-1}\|^2 \right\}.
\end{equation}

For given tolerance $\bar \varepsilon>0$, letting $\lam = \bar \varepsilon/(16M^2)$, then the iteration-complexity for PDS$(\hat x_0,\lam)$ to generate a primal-dual pair such that the primal-dual gap of a constrained version of \eqref{eq:problem} is bounded by $\bar \varepsilon$
is ${\cal O}(M^2 d_0^2/\bar \varepsilon^2)$ (see Theorem \ref{thm:pds}).

\subsection{Primal-dual cutting-plane scheme}\label{subsec:PDCP}
The PDPB method solves a sequence of proximal subproblems of the form
\begin{equation}\label{eq:phi-lam}
    \min_{u \in \R^n} \left\{\phi^\lam(u) :=\phi(u)+\frac{1}{2 \lambda}\left\|u-\hat x_{k-1}\right\|^2\right\},
\end{equation}
where $\lam$ is the prox stepsize and $\hat x_{k-1}$ is the prox center in the $k$-th proximal subproblem (or cycle). We omit the index $k$ in $\phi^\lam$ since the prox center is always fixed to be $\hat x_{k-1}$ in this subsection.
Each proximal subproblem invokes the PDCP scheme to find an approximate solution. Hence, PDPB can be viewed as a generalization of PDS, which only takes one proximal subgradient step (i.e., \eqref{eq:pds}) to solve every proximal subproblem \eqref{eq:phi-lam}.
The goal of this subsection is to describe the key subroutine PDCP for solving \eqref{eq:phi-lam} and present its primal-dual convergence analysis.

In the rest of this subsection, we consider subproblem \eqref{eq:phi-lam} with fixed prox center $\hat x_{k-1}$. For simplicity, we denote $\hat x_{k-1}$ as $x_0$ from a local perspective within the current cycle, as it is also the initial point of PDCP.  At the $j$-th iteration of PDCP, given some prox stepsize $\lam>0$ and prox center $x_0$, PDCP computes a primal-dual pair $(x_j,s_j)$ as follows
\begin{equation}\label{eq:xj}
    x_j=\underset{u \in \R^n}{\argmin}\left\{\Gamma_j(u) + h(u) +\frac{1}{2 \lam}\|u-x_0\|^2\right\}, \quad s_j \in \partial \Gamma_j(x_j) \cap (-\partial h^\lam(x_j)),  
\end{equation}
where $\Gamma_j$ is a proper, closed and convex function satisfying $\Gamma_j \le f$ for every $j\ge 1$, and
\begin{equation}\label{eq:hlam}
    h^\lam (\cdot) := h(\cdot) + \frac{1}{2\lam}\|\cdot - x_0\|^2.
\end{equation}
Starting from $\Gamma_1(\cdot) = \ell_f(\cdot ; x_0)$, and for every $j\ge 1$,  $\Gamma_{j+1}$ is obtained from the following generic bundle management (GBM), which is motivated by BU given in Subsection~3.1 of \cite{liang2024unified}.
It is easy to verify that the one-cut, two-cut, and multiple-cut schemes, denoted as (E1), (E2), and (E3) in Subsection 3.1 of \cite{liang2024unified}, all satisfy GBM.

\begin{algorithm}[H]
\caption{Generic Bundle Management, GBM$(\lambda, \tau_j, x_0, x_{j}, \Gamma_j)$}\label{alg:GBM}
\begin{algorithmic}
\REQUIRE $(\lam, \tau_j)\in \R_{++} \times [0,1]$, $(x_0,x_j) \in \R^n \times \R^n$, and $\Gamma_j \in \bConv{n}$ satisfying $\Gamma_j \le f$

\STATE $\bullet$ find a bundle model $\Gamma_{j+1}\in \bConv{n}$ satisfying $\Gamma_{j+1} \le f$ and 
        \begin{equation}\label{def:Gammaj}
	     \Gamma_{j+1}(\cdot) \ge \tau_j \bar \Gamma_j(\cdot) + (1-\tau_j)\ell_f(\cdot; x_j),
	\end{equation}
where $\bar \Gamma_j \in \bConv{n}$ satisfies $\bar \Gamma_j \le f$ and
        \begin{equation}\label{def:bar Gamma}
            \bar \Gamma_j(x_j) = \Gamma_j(x_j), \quad x_j = \underset{u\in \R^n}\argmin \left \{\bar \Gamma_j(u) + h(u) + \frac{1}{2\lam}\|u-x_0\|^2\right\}.
        \end{equation}

\ENSURE $\Gamma_{j+1}$.

\end{algorithmic} 
\end{algorithm}

PDCP computes an auxiliary sequence $\{\tx_j\}$ to determine termination. 
It generated $\tx_j$ such that
\begin{equation}\label{ineq:txj}
    \tilde x_1 = x_1, \quad \mbox{and} \quad \phi^\lam(\tx_{j+1}) \le \tau_j\phi^\lam(\tx_j) + (1-\tau_j)\phi^\lam(x_{j+1}), \, \forall j\ge 1,
\end{equation}
where $\phi^\lam$ is as in \eqref{eq:phi-lam}. 
PDCP also computes
\begin{equation}\label{eq:mj}
    m_j =\underset{u \in \R^n}{\min}\left\{\Gamma_j(u) + h(u) +\frac{1}{2 \lam}\|u-x_0\|^2\right\}, \quad t_j=\phi^\lam \left(\tilde x_j\right)-m_j.
\end{equation}
For given tolerance $\varepsilon>0$, PDCP terminates the current cycle when $t_j \le \varepsilon$.

PDCP is formally stated below.

\begin{algorithm}[H]
\caption{Primal-Dual Cutting-Plane, PDCP$(x_0,\lam,\varepsilon)$}\label{alg:PDCP}
\begin{algorithmic}
\REQUIRE given $x_0\in \dom h$, $\lam>0$, $\varepsilon>0$, set $t_0 = 2\varepsilon$, $\Gamma_1(\cdot)=\ell_f(\cdot;x_0)$, and $j=1$;

\WHILE{$t_{j-1} > \varepsilon$}

\STATE {\bf 1.} compute $(x_j,s_j)$ by \eqref{eq:xj}, choose $\tilde x_j$ as in \eqref{ineq:txj}, and set $t_j$ as in~\eqref{eq:mj};
\STATE {\bf 2.} select $\tau_j \in [0,1]$ and update $\Gamma_{j+1}$ by GBM$(\lambda, \tau_j, x_0, x_{j}, \Gamma_j)$ and $j \leftarrow j+1$;

\ENDWHILE
\ENSURE $(x_{j-1}, \tilde x_{j-1},  s_{j-1})$.

\end{algorithmic} 
\end{algorithm}

The auxiliary iterate $\tx_j$ vaguely given in \eqref{ineq:txj} can be explicitly computed by either of the following two formulas:
\[
\tilde x_{j+1} = \tau_j \tilde x_j + (1-\tau_j)x_{j+1}, \quad \forall j \ge 1,
\]
and 
\[
\tilde x_j \in \Argmin\{\phi^\lam(u): u \in\{x_1, \ldots, x_j\} \}, \quad \forall j \ge 1.
\]
Clearly, $\{\tilde x_j\}$ obtained from the second formula above satisfies \eqref{ineq:txj} with any $\tau_j \in [0,1]$.

The following result proves that $t_j$ is an upper bound on the primal-dual gap for \eqref{eq:phi-lam} and hence shows that $(\tx_j,s_j)$ an approximate primal-dual solution pair for \eqref{eq:phi-lam}.

\begin{lemma}\label{lem:tj}
    For every $j \ge 1$, we have
    \begin{equation}\label{ineq:gap}
        \phi^\lam\left(\tilde x_j\right) + f^*(s_j) + (h^\lam)^*(-s_j) \le t_j.
    \end{equation}
\end{lemma}

\begin{proof}
    It follows from \eqref{eq:xj} that $s_j \in \partial \Gamma_j(x_j)$ and $-s_j\in \partial h^\lam(x_j)$.
    Using Theorem 4.20 of \cite{Beck2017}, we have
    \[
    \Gamma_j^*(s_j) = -\Gamma_j(x_j) + \inner{s_j}{x_j}, \quad  (h^\lam)^*(-s_j) = - h^\lam(x_j)  - \inner{s_j}{x_j}.
    \]
    Combining the above identities and using the definition of $m_j$ in \eqref{eq:mj}, we have
    \[
    -m_j = \Gamma_j^*(s_j) + (h^\lam)^*(-s_j).
    \]
    It clearly from $\Gamma_j\le f$ that $\Gamma_j^*\ge f^*$.
    This observation and the above inequality imply that
    \[
    \phi^\lam\left(\tilde x_j\right) + f^*(s_j) + (h^\lam)^*(-s_j) \le \phi^\lam(\tilde x_j) - m_j.
    \]
    Hence, \eqref{ineq:gap} immediately follows from the definition of $t_j$ in \eqref{eq:mj}.
    Finally, we note that $- f^*(s) - (h^\lam)^*(-s)$ is the Lagrange dual function of $\phi^\lam(x)$ in \eqref{eq:phi-lam}. Therefore, the left-hand side of \eqref{ineq:gap} is the primal-dual gap for \eqref{eq:phi-lam}.
\end{proof}

With regard to Lemma \ref{lem:tj}, it suffices to show the convergence of $t_j$ to develop the primal-dual convergence analysis of PDCP.
We begin this analysis by providing some basic properties of GBM. 
The following result is adapted from Lemma~4.4 of \cite{liang2024unified}.

\begin{lemma}\label{lem:Gamma}
For every $j \ge 1$, there exists $\bar \Gamma_j \in \bConv{n}$ such that for every $u\in \R^n$,
\begin{equation}\label{ineq:bGammaj}
    \bar{\Gamma}_j(u) + h^\lam(u) \ge m_j + \frac{1}{2\lam}\|u-x_j\|^2.
\end{equation}
\end{lemma}

\begin{proof}
    Since the objective function in \eqref{def:bar Gamma} is $\lam^{-1}$-strongly convex, it follows from \eqref{def:bar Gamma} that 
    \[
    \bar \Gamma_j(u) + h(u) + \frac{1}{2\lam}\|u-x_0\|^2 \ge \bar \Gamma_j(x_j) + h(x_j) + \frac{1}{2\lam}\|x_j-x_0\|^2 + \frac{1}{2\lam}\|u-x_j\|^2.
    \]
    Inequality \eqref{ineq:bGammaj} immediately follows from the above inequality, the definition of $m_j$ in \eqref{eq:mj}, and the fact that $h^\lam(\cdot) = h(\cdot) + \|\cdot-x_0\|^2/(2\lam)$.
\end{proof}

Following Lemma \ref{lem:Gamma}, we are able to present the convergence rate of $t_j$ under the assumption that $\tau_j= j/(j+2)$ for every $j\ge 1$.
The following proposition resembles Lemma~4.6 in \cite{liang2024unified}.

\begin{proposition}\label{prop:inner}
Considering Algorithm \ref{alg:PDCP} with $\tau_j = j/(j+2)$, then for every $j \ge 1$, we have
\begin{equation}\label{ineq:tj}
    t_j\le \frac{2t_1}{j(j+1)} + \frac{16\lam M^2}{j+1},
\end{equation}
where $t_j$ is as in \eqref{eq:mj}.
Moreover, the number of iterations for PDCP to obtain $t_j \le \varepsilon$ is at most
\[
{\cal O}\left(\frac{\sqrt{t_1}}{\sqrt{\varepsilon}} + \frac{\lam M^2}{\varepsilon} + 1\right).
\]
\end{proposition}

\begin{proof}
We first note that for every $j\ge 1$, $\tau_j=A_j/A_{j+1}$ where $A_{j+1}=A_j+j+1$ and $A_0=0$, i.e., $A_j=j(j+1)/2$ for every $j\ge 0$.
It follows from this observation, the definition of $m_j$ in \eqref{eq:mj}, and relation \eqref{def:Gammaj} that
\begin{align*}
    A_{j+1}m_{j+1} & \stackrel{\eqref{eq:mj}}{=} A_{j+1}(\Gamma_{j+1}+h^\lam)(x_{j+1}) \\
    &\stackrel{\eqref{def:Gammaj}}{\ge} A_j \left[(\bar \Gamma_j+ h^\lam)(x_{j+1}) \right] + (j+1)\left[\ell_f(x_{j+1};x_j) + h^\lam(x_{j+1})\right].
\end{align*}
Applying Lemma~\ref{lem:Gamma} in the above inequality and using \eqref{ineq:f}, we have
\begin{align*}
    A_{j+1}m_{j+1} 
    &\stackrel{\eqref{ineq:bGammaj}}{\ge} A_j \left[ m_j+\frac{1}{2\lam}\|x_{j+1} - x_j\|^2 \right] + (j+1)\left[\ell_f(x_{j+1};x_j) + h^\lam(x_{j+1})\right] \\
    &= A_j m_j + (j+1)\left[\ell_f(x_{j+1};x_j) + h^\lam(x_{j+1}) + \frac{A_j}{2\lam (j+1)}\|x_{j+1} - x_j\|^2 \right] \\
    &\stackrel{\eqref{ineq:f}}{\ge} A_j m_j + (j+1)\left[\phi^\lam(x_{j+1}) - 2M\|x_{j+1}-x_j\| + \frac{A_j}{2\lam (j+1)}\|x_{j+1} - x_j\|^2 \right] \\
    &\ge A_j m_j + (j+1)\phi^\lam(x_{j+1}) - \frac{2\lam M^2 (j+1)^2}{A_j}
\end{align*}
where the last inequality is due to the Young's inequality $a^2+b^2\ge 2ab$.
It follows from the fact that $A_j=j(j+1)/2$ that for every $j\ge 1$,
\[
    A_{j+1}m_{j+1} \ge A_j m_j + (j+1)\phi^\lam(x_{j+1}) - 8\lam M^2.
\]
Replacing the index $j$ in the above inequality by $i$, summing the resulting inequality from $i=1$ to $j-1$, and using the definition of $t_j$ in \eqref{eq:mj} and the fact that $\tx_1=x_1$, we obtain
\begin{align*}
    A_j m_j
    &\ge A_1 m_1 + 2 \phi^\lam(x_2) + \dots + j\phi^\lam(x_j) - 8\lam M^2 (j-1) \\
    &\stackrel{\eqref{eq:mj}}{=} -A_1 t_1 + A_1 \phi^\lam(x_1) + 2 \phi^\lam(x_2) + \dots + j\phi^\lam(x_j) - 8\lam M^2 (j-1) \\
    &\stackrel{\eqref{ineq:txj}}{\ge} -A_1 t_1 + A_j \phi^\lam(\tx_j) - 8\lam M^2 (j-1),
\end{align*}
where the last inequality follows from \eqref{ineq:txj} and the fact that $A_j=A_{j-1}+j$.
Rearranging the terms and using the definition of $t_j$ in \eqref{eq:mj} again, we have
\begin{equation}\label{ineq:Aj-tj}
    A_j t_j \le A_1 t_1 + 8\lam M^2 (j-1).
\end{equation}
Hence, \eqref{ineq:tj} follows from the fact that $A_j=j(j+1)/2$.
Finally, the complexity bound immediately follows from \eqref{ineq:tj}.
\end{proof}

\subsection{Primal-dual proximal bundle method}\label{subsec:PDPB}

Recall the definitions of $d_0$ and $x_0^*$ in \eqref{eq:d0}. Since $x_0^* \in B(\hat x_0, 6d_0)$, which is the ball centered at $\hat x_0$ and with radius $6d_0$, it is easy to see that to solve \eqref{eq:problem}, it suffices to solve
\begin{equation}\label{eq:new}
    \min \left\{\hat \phi(x):=f(x)+\hat h(x): x \in \R^n\right\}
    = \min \left\{\phi(x): x \in Q\right\},
\end{equation}
where $\hat h = h + I_Q$ and $I_Q$ is the indicator function of $Q = B(\hat x_0, 6d_0)$.

In what follows, we present the PDPB and establish the complexity for obtaining a primal-dual solution pair of \eqref{eq:new}.
The PDPB is formally stated below.

\begin{algorithm}[H]
\caption{Primal-Dual Proximal Bundle, PDPB$(\hat x_0,\lam,\bar \varepsilon)$}\label{alg:PDPB}
\begin{algorithmic}
\REQUIRE given $(\hat x_0,\lam,\bar \varepsilon) \in \dom h \times \R_{++} \times \R_{++}$
\FOR{$k=1,2,\cdots$}
\STATE $\bullet$ call oracle $(\hat x_k,\tx_k,s_k) = \text{PDCP}(\hat x_{k-1},\lam, \bar \varepsilon)$ and compute 
\begin{equation}\label{eq:bxk}
    \bar x_k = \frac{1}{k}\sum_{i=1}^{k} \tx_i , \quad \bar s_k = \frac{1}{k}\sum_{i=1}^{k} s_i.
\end{equation}
\ENDFOR
\end{algorithmic} 
\end{algorithm}

In the $k$-th iteration of PDPB, we are approximately solving the proximal subproblem \eqref{eq:phi-lam}.
More specifically, the pair $(\tx_k,s_k)$ is a primal-dual solution to \eqref{eq:phi-lam} with the primal-dual gap bounded by $\bar \varepsilon$ (see Lemma \ref{lem:tj}).
Recall from Subsection~\ref{subsec:PDCP} that \eqref{eq:phi-lam} is approximately solved by invoking PDCP. The (global) iteration indices in PDCP are regarded as the $k$-th cycle, denoted by ${\cal C}_k=\left\{i_k, \ldots, j_k\right\}$, where $j_k$ is the last iteration index of the $k$-th call to PDCP, $j_0=0$, and $i_k = j_{k-1}+1$.
Hence, for the $j_k$-th iteration of PDCP, we have
\begin{equation}\label{eq:translate}
    \hat x_k = x_{j_k}, \quad \tilde x_k = \tilde x_{j_k}, \quad s_k = s_{j_k}, \quad \Gamma_k  = \Gamma_{j_k}, \quad m_k = m_{j_k},
\end{equation}
and \eqref{eq:xj} becomes
\begin{equation}\label{eq:xk}
    \hat x_k=\underset{u \in \R^n}{\argmin}\left\{\Gamma_k(u) + h(u) +\frac{1}{2 \lambda}\left\|u-\hat x_{k-1}\right\|^2\right\}, \quad s_k \in \partial \Gamma_k(\hat x_k) \cap (-\partial h^\lam(\hat x_k)).
\end{equation}

The following lemma provides basic properties of PDPB and is the starting point of the the complexity analysis of PDPB.

\begin{lemma}\label{ser-lemma1}
    The following statements hold for every $k \ge 1$:
    \begin{enumerate}
        \item [(a)] $\Gamma_k \le f$ and $f^*\le \Gamma_k^*$;
        \item [(b)] $s_k \in \partial \Gamma_k(\hat x_k)$ and $g_k \in \partial h(\hat x_k)$ where $g_k = - s_k + (\hat x_{k-1}-\hat x_k)/\lam$;
        \item [(c)] $\phi^\lam(\tilde x_k) \le \bar \varepsilon + m_k = \bar \varepsilon + (\Gamma_k + h)(\hat x_k) + \|\hat x_k - \hat x_{k-1}\|^2/(2\lam)$.
    \end{enumerate}
\end{lemma}

\begin{proof}
    (a) It follows from the facts that $\Gamma_j \le f$ for every $j\ge 1$ and $\Gamma_k=\Gamma_{j_k}$ that the first inequality holds. The second one immediately follows from the first one and the definition of the conjugate function.

    (b) This statement follows from \eqref{eq:xk} and the definitions in \eqref{eq:translate}.

    (c) This statement follows from the termination criterion of the $k$-th cycle, that is, $t_{j_k}\le \bar \varepsilon$, and the definitions in \eqref{eq:mj} and \eqref{eq:translate}.
\end{proof}

The following proposition is a key component of our complexity analysis, as it establishes an important primal-dual gap for \eqref{eq:problem}.

\begin{proposition}\label{prop:pd-key}
For every $k \ge 1$, we have
\begin{equation}\label{ineq:conjugate}
    \phi(\bar{x}_k) + f^*(\bar s_k) + \hat h ^*(-\bar s_k) \le \bar \varepsilon + \frac{18 d_0^2}{\lam k}.
\end{equation}
where $\bar x_k$ and $\bar s_k$ are as in \eqref{eq:bxk}.
\end{proposition}

\begin{proof}
It follows from Lemma \ref{ser-lemma1}(b) and Theorem 4.20 of \cite{Beck2017} that for every $k \ge 1$,
\[ 
\Gamma_k(\hat x_k) +  \Gamma_k^*( s_k) = \inner{\hat x_k}{s_k}, \quad  h(\hat x_k) +  h^*(g_k) = \inner{\hat x_k}{g_k}.
\]
Summing the above two equations yields
\begin{equation}\label{eq:Gamma-h}
    (\Gamma_k + h)(\hat x_k) + {\Gamma}_k^*(s_k) + h^*(g_k) = \frac{1}{\lam}\inner{\hat x_k}{\hat x_{k-1}-\hat x_k}.
\end{equation}
Using the above identity and Lemma \ref{ser-lemma1}(a) and (c), we have for every $k \ge 1$,
\begin{align*}
& \phi(\tilde x_k) + f^*(s_k) +  {h}^*(g_k) \le \phi(\tilde x_k) + \Gamma_k^*(s_k) +  {h}^*(g_k) \\
\le & \bar \varepsilon +(\Gamma_k + h)(\hat x_k) + \frac{1}{2\lam}\|\hat x_k-\hat x_{k-1}\|^2 + \Gamma_k^*(s_k) + {h}^*(g_k)\\
\stackrel{\eqref{eq:Gamma-h}}{=}& \bar \varepsilon + \frac{1}{2\lam}(\|\hat x_{k-1}\|^2-\|\hat x_k\|^2).
\end{align*}
Replacing the index $k$ in the above inequality by $i$, summing the resulting inequality from $i=1$ to $k$, and using convexity and the definitions in \eqref{eq:bxk}, we obtain 
\begin{equation}\label{ineq:avg}
    \phi(\bar x_k) + f^*(\bar s_k) + h^*(\bar g_k)\le \bar \varepsilon + \frac{1}{2\lam k }\left(\|\hat x_0\|^2-\|\hat x_k\|^2\right),
\end{equation}
where $\bar g_k = (\sum_{i=1}^k g_i)/k$.
Define
\begin{equation}\label{eq:zeta}
    \eta_k(u) = \frac{1}{2\lam k} \|u-\hat x_0\|^2, \quad \hat \eta_k(u) = \eta_k(u) - I_Q(u).
\end{equation}
Noting that $\nabla \eta_k(\hat x_k) = (\hat x_k - \hat x_0)/(\lam k) = -\bar g_k - \bar s_k$, and hence it follows from Theorem 4.20 of \cite{Beck2017} that
\[
\eta_k^*(-\bar g_k - \bar s_k) = \frac{1}{\lam k}\inner{\hat x_k-\hat x_0}{\hat x_k} - \eta_k(\hat x_k) = \frac{1}{2\lam k}\left(\|\hat x_k\|^2 - \| \hat x_0\|^2\right).
\]
The above observation and \eqref{ineq:avg} together imply that
\begin{equation}\label{ineq:bar}
    \phi(\bar x_k) + f^*(\bar s_k) + h^*(\bar g_k) + \eta_k^*(-\bar g_k - \bar s_k) \le \bar \varepsilon.
\end{equation}
It follows from Theorem 4.17 of \cite{Beck2017} and the definition of infimum convolution in \eqref{def:infconv} that
\[
(h + \eta_k)^*(-\bar s_k) = (h^* \square \eta_k^*)(-\bar s_k)
\stackrel{\eqref{def:infconv}}= \min_{u\in \R^n}\{h^*(u) + \eta_k^*(-\bar s_k-u)\}
\le h^*(\bar g_k) + \eta_k^*(-\bar g_k - \bar s_k).
\]
Noting from \eqref{eq:zeta} that $\hat h = h + \eta_k -\hat \eta_k$ and applying Theorem 4.17 of \cite{Beck2017} again, we obtain
\begin{align*}
    \hat h ^*(-\bar s_k) &= [(h + \eta_k)^* \square (-\hat \eta_k)^*](-\bar s_k)
= \min_{u\in \R^n}\{(h + \eta_k)^*(u) + (-\hat \eta_k)^*(-\bar s_k-u)\}\\
&\le (h + \eta_k)^*(-\bar s_k) + (-\hat \eta_k)^*(0).
\end{align*}
Summing the above two inequalities, we have
\[
\hat h ^*(-\bar s_k) \le h^*(\bar g_k) + \eta_k^*(-\bar g_k - \bar s_k) + (-\hat \eta_k)^*(0),
\]
which together with \eqref{ineq:bar} implies that
\[
\phi(\bar x_k) + f^*(\bar s_k) + \hat h ^*(-\bar s_k) \le \bar \varepsilon + (-\hat \eta_k)^*(0).
\]
It follows from \eqref{eq:zeta} that
\[
(-\hat \eta_k)^*(0) = \max_{u \in \R^n}\left\{\inner{0}{u} - \left( -\frac{\|u-\hat x_0\|^2}{2\lam k} +I_Q(u)\right)\right\}
=\frac{\max_{u \in Q}\|u-\hat x_0\|^2}{2\lam k}
= \frac{18d_0^2}{\lam k},
\]
where the last identity follows from the fact that $Q = B(\hat x_0, 6d_0)$.
Therefore, \eqref{ineq:conjugate} holds in view of the above two relations.
\end{proof}

The next lemma is a technical result showing that $\hat x_k \in Q$ and $\tx_k \in Q$ under mild conditions, where $Q = B(\hat x_0,6d_0)$.

\begin{lemma}\label{lem:bound}
    Given $(\hat x_0,\bar \varepsilon) \in \R^n \times \R_{++}$, if $\lam \le 2d_0^2/\bar \varepsilon$ and $k \le 2d_0^2/(\lam \bar \varepsilon)$, then the sequences $\{\hat x_k\}$ and $\{\tx_k\}$ generated by PDPB$(\hat x_0,\lam, \bar \varepsilon)$ satisfy
    \begin{equation}\label{incl:xk}
        \hat x_k \in Q, \quad \tx_k \in Q.
    \end{equation}
\end{lemma}

\begin{proof}
    Noticing that the objective function in \eqref{eq:xk} is $\lam^{-1}$-strongly convex, it thus follows from Theorem 5.25(b) of \cite{Beck2017} that for every $u \in \dom h$,
    \begin{equation}\label{ineq:key}
        m_k + \frac{1}{2\lam} \|u-\hat x_{k}\|^2 \le \Gamma_k(u) + h(u) + \frac1{2\lam}\|u-\hat x_{k-1}\|^2 \le \phi(u)+ \frac1{2\lam}\|u-\hat x_{k-1}\|^2, 
    \end{equation}
    where the second inequality follows from the first one in Lemma \ref{ser-lemma1}(a).
    Taking $u=x_0^*$ in \eqref{ineq:key}, we have
    \[
    m_k + \frac{1}{2\lam} \|\hat x_k-x_0^*\|^2 \le \phi_* + \frac1{2\lam}\|\hat x_{k-1}-x_0^*\|^2. 
    \]
    This inequality and Lemma \ref{ser-lemma1}(c) then imply that
    \begin{align*}
        \frac{1}{2\lam}\|\hat x_k - x_0^*\|^2 &\le \phi(\tx_k)- \phi_* + \frac{1}{2\lam}\|\hat x_k - x_0^*\|^2 \\
        &\le \phi(\tx_k) - m_k + \frac{1}{2\lam}\|\hat x_{k-1}-x_0^*\|^2 \le \bar \varepsilon + \frac{1}{2\lam}\|\hat x_{k-1}-x_0^*\|^2.
    \end{align*}
    Replacing the index $k$ in the above inequality by $i$ and summing the resulting inequality from $i=1$ to $k$, we have
    \[
    \|\hat x_k - x_0^*\|^2 \le \|\hat x_0-x_0^*\|^2 + 2 k \lam \bar \varepsilon.
    \]
    Using the fact that $\sqrt{a+b} \le \sqrt{a} + \sqrt{b}$ for $a,b \ge 0$ and the assumption that $k \le 2d_0^2/(\lam \bar \varepsilon)$, we further obtain
    \begin{equation}\label{ineq:xk}
        \|\hat x_k - x_0^*\| \le d_0 + \sqrt{2k\lam \bar \varepsilon} \le 3d_0.
    \end{equation}
    Taking $u=\tx_k$ in \eqref{ineq:key} and using Lemma \ref{ser-lemma1}(c), we have
    \[
    \frac{1}{2\lam} \|\tx_k-\hat x_k\|^2 \le \phi(\tx_k) + \frac1{2\lam}\|\tx_k - \hat x_{k-1}\|^2 - m_k \le \bar \varepsilon.
    \]
    Under the assumption that $\lam \le 2d_0^2/\bar \varepsilon$, using \eqref{ineq:xk}, the above inequality, and the triangle inequality, we have
    \begin{align*}
        \|\hat x_k - \hat x_0\| &\le \|\hat x_k - x_0^*\| + \|x_0^* - \hat x_0\| \stackrel{\eqref{ineq:xk}}{\le} 4d_0, \\
        \|\tx_k - \hat x_0\| &\le \|\hat x_k - \hat x_0\| + \|\tx_k - \hat x_k\| \le 4d_0 + \sqrt{2\lam \bar \varepsilon} \le 6d_0.
    \end{align*}
    Hence, \eqref{incl:xk} follows immediately.
\end{proof}

Now we are ready to present the number of oracle calls to PDCP in PDPB (i.e., Algorithm~\ref{alg:PDPB}).

\begin{proposition}\label{prop:pdpb}
    Given $(\hat x_0,\bar \varepsilon) \in \R^n \times \R_{++}$, if $\lam \le 2d_0^2/\bar \varepsilon$, then the number of iterations for PDPB$(\hat x_0,\lam, \bar \varepsilon)$ to generate $(\bar x_k, \bar s_k)$ satisfying 
    \begin{equation}\label{ineq:pd}
        \hat \phi(\bar x_k) + f^*(\bar s_k) + \hat h^*(-\bar s_k) \le 10\bar \varepsilon
    \end{equation}
    is at most $2 d_0^2/(\lam \bar \varepsilon)$.
\end{proposition}

\begin{proof}
    Since $Q$ is a convex set, it follows from the definition of $\bar x_k$ in \eqref{eq:bxk} and Lemma \ref{lem:bound} that $\bar x_k \in Q$ for every $k \le 2d_0^2/(\lam \bar \varepsilon)$.
    This observation and the fact that $\hat h = h+ I_Q$ imply that $\hat h(\bar x_k) = h(\bar x_k)$.
    Hence, using Proposition~\ref{prop:pd-key}, we have for every $k \le 2d_0^2/(\lam \bar \varepsilon)$,
    \[
    \hat \phi(\bar x_k) + f^*(\bar s_k) + \hat h ^*(-\bar s_k) \le \bar \varepsilon + \frac{18 d_0^2}{\lam k}.
    \]
    Therefore, the conclusion of the proposition follows immediately.
\end{proof}

The following lemma is a technical result providing a universal bound on the first gap $t_{i_k}$ for each cycle ${\cal C}_k$.

\begin{lemma}\label{lem:t1}
For $k \le 2d_0^2/(\lam \bar \varepsilon)$, we have
\begin{equation}\label{ineq:t1}
    t_{i_k} \le \bar{t}:= 4M(3d_0 + \lam M),
\end{equation}
where $i_k$ is the first iteration index in the cycle ${\cal C}_k$.
\end{lemma}

\begin{proof}
Using \eqref{ineq:f}, definitions of $m_j$ and $t_j$ in \eqref{eq:mj}, and the facts that $\tx_{i_k}=x_{i_k}$ and $\Gamma_{i_k}=\ell_f(\cdot;x_{k-1})$, we have
\begin{align}
    t_{i_k} &\stackrel{\eqref{eq:mj}}{=} \phi^\lam \left(\tx_{i_k}\right)-m_{i_k} = \phi^\lam (x_{i_k}) -m_{i_k} \nn \\
    &\stackrel{\eqref{eq:problem},\eqref{eq:phi-lam},\eqref{eq:mj}}{=} f(x_{i_k})-\ell_f(x_{i_k};\hat x_{k-1}) \stackrel{\eqref{ineq:f}}{\le} 2M\left\|x_{i_k}-\hat x_{k-1}\right\|, \label{ineq:tik}
\end{align}
where we have also used the definitions of $\phi$ and $\phi^\lam$ in \eqref{eq:problem} and \eqref{eq:phi-lam}, respectively, in the last identity.
In view of \eqref{eq:xj} and the fact that $\Gamma_{i_k}=\ell_f(\cdot;\hat x_{k-1})$, we know the first iteration of PDCP is the same as PDS$(\hat x_0,\lam)$ (see \eqref{eq:pds}). Hence, following an argument similar to the proof of Lemma~\ref{pds-lemma1}, we can prove for every $u \in \dom h$,
\[
\phi(x_{i_k}) - \ell_f(u;\hat x_{k-1}) -h(u)
\stackrel{\eqref{ineq:hat-phi-recur}}{\le} 2\lam M^2 + \frac{1}{2 \lambda}\|u -\hat x_{k-1}\|^2 - \frac{1}{2\lambda}\left\|u-x_{i_k}\right\|^2.
\]
It follows from the above inequality with $u=x_0^*$ and the convexity of $f$ that
\begin{align*}
    0 &\le \phi(x_{i_k}) -\phi_*
    \le \phi(x_{i_k}) - \ell_f(x_0^*;\hat x_{k-1}) -h(x_0^*) \\
    &\le 2\lam M^2 + \frac{1}{2 \lambda}\|x_0^* -\hat x_{k-1}\|^2 - \frac{1}{2\lambda}\left\|x_0^*-x_{i_k}\right\|^2.
\end{align*}
Rearranging the terms and using the inequality $\sqrt{a+b} \le \sqrt{a} + \sqrt{b}$ for any $a, b \ge 0$, we have
\[
\|x_0^*-x_{i_k}\| \le\|x_0^*-\hat x_{k-1}\|+2\lam M.
\]
This inequality and the triangle inequality then imply that
\[
\|x_{i_k}-\hat x_{k-1}\| \le \|x_{i_k}-x_0^*\|+\|x_0^*-\hat x_{k-1}\| \le 2\|\hat x_{k-1}-x_0^*\|+2 \lam M.
\]
Recall from the proof of Lemma \ref{lem:bound} that \eqref{ineq:xk} gives $\|\hat x_k-x_0^*\| \le 3d_0$ for $k \le 2d_0^2/(\lam \bar \varepsilon)$.
Hence, we have
\[
\|x_{i_k}-\hat x_{k-1}\| \le 2(3d_0 + \lam M).
\]
Therefore, \eqref{ineq:t1} follows from \eqref{ineq:tik} and the above inequality.
\end{proof}

Finally, we are ready to establish the total iteration-complexity of PDPB.

\begin{theorem}
    Given $(\hat x_0,\bar \varepsilon) \in \R^n \times \R_{++}$, assuming that $\lam$ satisfies
    \begin{equation}\label{ineq:lam}
        \frac{\sqrt{\bar \varepsilon d_0}}{M^{3/2}} \le \lam \le \frac{2d_0^2}{\bar \varepsilon},
    \end{equation}
    then the total iteration-complexity of PDPB$(\hat x_0,\lam,\bar \varepsilon)$ to find $(\bar x_k,\bar s_k)$ satisfying \eqref{ineq:pd} is
    \begin{equation}\label{cmplx:PDPB}
        {\cal O}\left( \frac{M^2 d_0^2}{\bar \varepsilon^2} + 1\right).
    \end{equation}
\end{theorem}

\begin{proof}
    In view of Proposition \ref{prop:pdpb}, PDPB takes 
    \begin{equation}\label{cmplx:outer}
        {\cal O}\left(\frac{d_0^2}{\lam \bar \varepsilon}+1\right)
    \end{equation}
    cycles to find $(\bar x_k,\bar s_k)$ satisfying \eqref{ineq:pd}.
    It follows from Proposition \ref{prop:inner} and Lemma \ref{lem:t1} that for every cycle in PDPB before termination, the number of iterations in the cycle is 
    \[
    {\cal O}\left(\frac{\sqrt{Md_0 + \lam M^2}}{\sqrt{\bar \varepsilon}} + \frac{\lam M^2}{\bar \varepsilon} + 1\right)=
    {\cal O}\left(\frac{\sqrt{Md_0}}{\sqrt{\bar \varepsilon}} + \frac{\lam M^2}{\bar \varepsilon} + 1\right),
    \]
    which together with the assumption that $\sqrt{\bar \varepsilon d_0}/M^{3/2} \le \lam$ becomes
    \begin{equation}\label{cmplx:inner}
        {\cal O}\left(\frac{\lam M^2}{\bar \varepsilon} + 1\right).
    \end{equation}
    Combining \eqref{cmplx:outer} and \eqref{cmplx:inner}, and using \eqref{ineq:lam}, we conclude that \eqref{cmplx:PDPB} holds.
\end{proof}

\section{Duality between PDCP and CG}\label{sec:FW}

The dual problem of the proximal subproblem \eqref{eq:phi-lam} can be written as
\begin{equation}\label{eq:dual}
    \min_{z\in\R^n}\left\{\psi(z): = (h^\lam)^*(-z) + f^*(z) \right\},
\end{equation}
where $-\psi$ is the dual function of $\phi^\lam$ given by \eqref{eq:phi-lam} and $h^\lam$ is as in \eqref{eq:hlam}. Since $h^\lam$ is $\lam^{-1}$-strongly convex, $(h^\lam)^*$ is $\lam$-smooth and one possible algorithm to solve \eqref{eq:dual} is the CG method.

We describe CG for solving \eqref{eq:dual} below.

\begin{algorithm}[H]
\caption{Conditional Gradient for \eqref{eq:dual}, CG$(z_1)$}
\begin{algorithmic}\label{alg:CG}
\REQUIRE  given $z_1 \in \dom f^* $
\FOR{$j=1,2,\cdots$}
\STATE
\vspace{-8mm}
\begin{align}
    & \bar z_j  =  \underset{z \in \R^n}\argmin\left\{\inner{-\nabla ( h ^\lam)^* (-z_j)}{z} + f^*(z)\right\}, \label{eq:bar-zj}\\
    & z_{j+1} = \tau_j z_j + (1-\tau_j) \bar z_j.\label{eq:z+} 
\end{align}
\vspace{-8mm}
\ENDFOR
\end{algorithmic} 
\end{algorithm}

Motivated by the duality between the mirror descent/subgradient method and CG studied in \cite{bach2015duality}, we prove the nice connection between CG (i.e., Algorithm \ref{alg:CG}) and PDCP (i.e., Algorithm \ref{alg:PDCP}) via duality.
More specifically, we consider a specific implementation of GBM within PDCP, that is $\Gamma_j$ is updated as
\begin{equation}\label{eq:one-cut}
    \Gamma_{j+1}(\cdot) = \tau_j \Gamma_j(\cdot) + (1-\tau_j)\ell_f(\cdot; x_j).
\end{equation}
Since $\Gamma_1(\cdot) = \ell_f(\cdot;x_0)$, $\Gamma_j$ is always affine and $s_j = \nabla \Gamma_j$ in view of \eqref{eq:xj}. 

The following result reveals the duality between PDCP with update scheme \eqref{eq:one-cut} and CG. Since the tolerance $\bar \varepsilon$ is not important in the discussion below, we will ignore it as input to PDCP. Assuming $\lam$ in both PDCP and CG are the same, we only focus on the initial points of the two methods. Hence, we denote them by PDCP$(x_0)$ and CG$(z_1)$.

\begin{theorem}\label{thm:duality}
Given $x_0\in \R^n$, $z_1 = f'(x_0)$, and the sequence $\{\tau_j\}$, then PDCP$(x_0)$ with update scheme \eqref{eq:one-cut} for solving \eqref{eq:phi-lam} and CG$(z_1)$ for solving \eqref{eq:dual} have the following correspondence for every $j\ge 1$,
\begin{equation}\label{eq:correspond}
    s_j = z_j, \quad x_j = \nabla ( h ^\lam)^* (-z_j), \quad f'(x_j) = \bar z_j.
\end{equation}
\end{theorem}

\begin{proof}
We first show that the first relation in \eqref{eq:correspond} implies the other two in \eqref{eq:correspond}.
Using the definition of $x_j$ in \eqref{eq:xj}, the fact that $s_j=\nabla \Gamma_j$, and the first relation in \eqref{eq:correspond}, we have $x_j$ from PDCP is equivalent to
\[
x_j \stackrel{\eqref{eq:xj}}= \underset{x\in \R^n}\argmin \{\inner{s_j}{x} + h^\lam(x)\} \stackrel{\eqref{eq:correspond}}= \underset{x \in \R^n}\argmax\{\inner{-z_j}{x} - h^\lam(x)\},
\]
which implies that the second relation in \eqref{eq:correspond} holds.
The last one in \eqref{eq:correspond} similarly follows from the second relation and \eqref{eq:bar-zj}.

We next prove the first relation in \eqref{eq:correspond} by induction. For the case $j=1$, it is easy to see from $\Gamma_1(\cdot) = \ell_f(\cdot;x_0)$ that 
\[
s_1=\nabla \Gamma_1 = f'(x_0) = z_1.
\]
Assume that the first relation in \eqref{eq:correspond} holds for some $j\ge 1$. By the argument above, we know that the second and third relations in \eqref{eq:correspond} also hold for $j$.
Using the fact that $s_j=\nabla \Gamma_j$, the definition of $\Gamma_{j+1}$ in \eqref{eq:one-cut}, and the last two relations in \eqref{eq:correspond}, we obtain
\begin{align*}
    s_{j+1} &= \nabla \Gamma_{j+1} \stackrel{\eqref{eq:one-cut}}= \tau_j \nabla \Gamma_j + (1-\tau_j) f'(x_j) \\
    &\stackrel{\eqref{eq:correspond}}= \tau_j s_j + (1-\tau_j) \bar z_j \stackrel{\eqref{eq:correspond}}= \tau_j z_j + (1-\tau_j) \bar z_j \stackrel{\eqref{eq:z+}}= z_{j+1},
\end{align*}
where the last identity is due to \eqref{eq:z+}.
Hence, the first relation in \eqref{eq:correspond} also holds for the case $j+1$. We thus complete the proof.
\end{proof}

\subsection{Alternative primal-dual convergence analysis of PDCP} \label{subsec:alt-pf}

Theorem \ref{thm:duality} demonstrates that PDCP and CG represent primal and dual perspectives for solving the equivalent problems \eqref{eq:phi-lam} and \eqref{eq:dual}, respectively.
Recall that Proposition~\ref{prop:inner} establishes the primal-dual convergence rate of PDCP for solving \eqref{eq:phi-lam}, and hence it is worth studying the primal-dual convergence of CG for solving \eqref{eq:dual} as well. Thanks to the duality connection illustrated by Theorem \ref{thm:duality}, the convergence analysis of CG also serves as an alternative approach to study PDCP from the dual perspective. 

Recall from (13.4) of \cite{Beck2017} that the Wolfe gap $S:\R^n \to \R$ for problem \eqref{eq:dual} is defined by
\begin{equation}\label{eq:Wolfe}
    S(w)=\max _{z \in \R^n}\left\{-\inner{\nabla (h^\lam)^*(-w)}{w-z} + f^*(w) - f^*(z)\right\}.
\end{equation}
In the following lemma, we show that $S(z_j)$ is a primal-dual gap for \eqref{eq:dual}. This result is an analogue of Lemma \ref{lem:tj}, which also shows that $t_j$ is a primal-dual gap for \eqref{eq:phi-lam}.

\begin{lemma}\label{lem:Wolfe-dual}
Suppose that the assumptions in Theorem \ref{thm:duality} hold,  
then for every $j \ge1$, we have
\begin{equation}\label{eq:S-dual}
    S(z_j) = \phi^\lam(x_j) + \psi (z_j).
\end{equation}
\end{lemma}

\begin{proof}
Since the assumptions in Theorem \ref{thm:duality} hold, it follows from Theorem \ref{thm:duality} that \eqref{eq:correspond} holds for every $j\ge 1$.
Using the second relation in \eqref{eq:correspond} and the definition of $S(w)$ in \eqref{eq:Wolfe}, we have
\begin{align*}
    S(z_j)&\stackrel{\eqref{eq:Wolfe}}=\max _{z \in \R^n}\left\{-\inner{\nabla (h^\lam)^*(-z_j)}{z_j-z} + f^*(z_j) - f^*(z)\right\} \\
    &\stackrel{\eqref{eq:correspond}}= f^*(z_j) -\inner{x_j}{z_j} + \max _{z \in \R^n}\left\{\langle x_j, z\rangle - f^*(z)\right\} \\
    &= f^*(z_j) + \inner{x_j}{-z_j} + f(x_j) \\
    &\stackrel{\eqref{eq:correspond}}= f^*(z_j) + (h^\lam)^*(-z_j) + h^\lam(x_j) + f(x_j),
\end{align*}
where we use the second relation in \eqref{eq:correspond} again in the last identity.
Finally, \eqref{eq:S-dual} immediately follows from the definitions of $\phi^\lam$ and $\psi$ in \eqref{eq:phi-lam} and \eqref{eq:dual}, respectively.
\end{proof}

Recalling from Lemma \ref{lem:tj} and using the first relation in \eqref{eq:correspond} and the definition of $\psi$ in \eqref{eq:dual}, we know
\begin{equation}\label{ineq:tj-bound}
    t_j \ge \phi^\lam\left(\tilde x_j\right)  + \psi (z_j),
\end{equation}
i.e., $t_j$ an upper bound on a primal-dual gap for \eqref{eq:dual}.
On the other hand, Lemma \ref{lem:Wolfe-dual} shows that $S(z_j)$ is a primal-dual gap for \eqref{eq:dual}.
We also note that the primal iterate used in $S(z_j)$ is $x_j$, while the one used in $t_j$ is $\tx_j$.

The following lemma gives a basic inequality used in the analysis of CG, which is adapted from Lemma 13.7 of \cite{Beck2017}. For completeness, we present Lemma 13.7 of \cite{Beck2017} as Lemma \ref{lem:Beck-FW} in Appendix \ref{sec:technical}.

\begin{lemma}\label{lem:FW-basic}
For every $j \ge 1$ and $\tau_j \in [0,1]$, the iterates $z_j$ and $\bar z_j$ generated by Algorithm~\ref{alg:CG} satisfy
\begin{equation}\label{ineq:CG-fundamental}
    \psi(z_{j+1}) \le \psi(z_j) - (1-\tau_j) S(z_j) + \frac{(1-\tau_j)^2 \lam}{2}\| \bar z_j - z_j\|^2.
\end{equation}
\end{lemma}

\begin{proof}
It is easy to see that \eqref{eq:dual} as an instance of \eqref{eq:Beck} with
\[
F = \psi, \quad f = (h^\lam)^*, \quad g = f^*, \quad L_f = \lam.
\]
Therefore, \eqref{ineq:CG-fundamental} immediately follows from \eqref{ineq:Beck-basic} with
\[
x = z_j, \quad t= 1-\tau_j, \quad p(x) = \bar z_j, \quad x+t(p(x) -x) = z_{j+1}.
\]
\end{proof}

Define
\begin{equation}\label{def:uj}
    u_j = \begin{cases}    x_1, & \text{if} ~ j=1;\\
\tau_{j-1} u_{j-1} + (1-\tau_{j-1})x_{j-1}, & \text{otherwise.}
\end{cases}
\end{equation}
We are now ready to prove the primal-dual convergence of CG in terms of gap $\phi^\lam(u_j) + \psi(z_j)$ in the following theorem, which resembles Proposition \ref{prop:inner} for PDCP.
An implicit assumption is that we are solving \eqref{eq:dual} as the dual to the proximal subproblem \eqref{eq:phi-lam} within PDPB. Consequently, the iteration count $k$ in PDPB satisfies $k \le 2d_0^2/(\lam \bar \varepsilon)$, in accordance with the assumption in Lemma \ref{lem:t1}.

\begin{theorem}\label{thm:CG-converge}
Suppose that the assumptions in Theorem \ref{thm:duality} hold, and $\tau_j =j/(j+2)$, then for every $j \ge 1$,
\begin{equation}\label{ineq:CG-convergence}
    \phi^\lam(u_j) + \psi(z_j) \le \frac{8M(3d_0 + \lam M)}{j(j+1)} + \frac{8\lam M^2}{j+1}.
\end{equation}
\end{theorem}
\begin{proof}
Using Lemma \ref{lem:Wolfe-dual}, the convexity of $\phi^\lam$, and definition of $u_j$ in \eqref{def:uj}, we have for every $j\ge 1$,
\begin{align*}
- (1-\tau_j)S(z_j) & \stackrel{\eqref{eq:S-dual}}= - (1-\tau_j)\phi^\lam(x_j) - (1-\tau_j) \psi (z_j) \\
& \stackrel{\eqref{def:uj}}\le -  \phi^\lam(u_{j+1}) + \tau_j \phi^\lam(u_j) - (1-\tau_j)\psi(z_j).
\end{align*}
This inequality and Lemma \ref{lem:FW-basic} imply that
\[
\phi^\lam (u_{j+1}) + \psi(z_{j+1}) \stackrel{\eqref{ineq:CG-fundamental}}\le \tau_j [\phi^\lam (u_j) + \psi(z_j)] + 2(1-\tau_j)^2\lam M^2,
\]
where we also use the facts that $\|\bar z_j\|\le M$ and $\|z_j\|\le M$ due to (A2) and $\bar z_j, z_j \in \dom f^*$.
Note that for every $j\ge 1$, $\tau_j=A_j/A_{j+1}$ where $A_{j+1}=A_j+j+1$ and $A_0=0$, i.e., $A_j=j(j+1)/2$ for every $j\ge 0$.
It thus follows from the above inequality that
\[
A_{j+1}[\phi^\lam (u_{j+1}) + \psi(z_{j+1})] \le A_j [\phi^\lam (u_j) + \psi(z_j)] + 4\lam M^2.
\]
Replacing the index $j$ in the above inequality by $i$, summing the resulting inequality from $i=1$ to $j$, and using the fact that $A_1=1$, we obtain
\[
A_j[\phi^\lam (u_j) + \psi(z_j)] \le \phi^\lam (u_1) + \psi(z_1) + 4\lam M^2 j.
\]
In view of \eqref{def:uj}, it is easy to see that $u_1 = x_1 = \tx_1$, which together with Lemma \ref{lem:t1} and \eqref{ineq:tj-bound} yields that
\[
\phi^\lam (u_1) + \psi(z_1) = \phi^\lam(\tx_1) + \psi(z_1) \stackrel{\eqref{ineq:tj-bound}}{\le} t_1 \stackrel{\eqref{ineq:t1}}{\le} 4M(3d_0 + \lam M). 
\]
Therefore, \eqref{ineq:CG-convergence} immediately follows from the above two inequalities and the fact that $A_j=j(j+1)/2$.
\end{proof}

The results in this subsection justify the implementation of proximal subproblem \eqref{eq:phi-lam} using CG from the dual point of view. In other words, PDPB can be also understood as the inexact PPM with CG as a subroutine.

\subsection{GBM implementations inspired by CG} \label{subsec:new GBM}

The discussion in this section so far is based on a particular implementation of GBM within PDCP, i.e., the one-cut scheme \eqref{eq:one-cut} with $\tau_j=j/(j+2)$ for every $j \ge 1$. Note that $\tau_j=j/(j+2)$ is also a standard choice in CG but not the only option. Inspired by alternative choices of $\tau_j$ used in CG (e.g., Section 13.2.3 of \cite{Beck2017}), we also consider
\begin{equation}\label{eq:alphaj}
    \alpha_j = \max \left\{0, 1-\frac{S(z_j)}{\lam \|z_j-\bar z_j\|^2}\right\}
\end{equation}
and
\begin{equation}\label{eq:betaj}
    \beta_j \in \Argmin \left\{\psi(\beta z_j + (1-\beta) \bar z_j): \beta \in[0,1]\right\}
\end{equation}
in this subsection and establish convergence rates of CG as in Theorem \ref{thm:CG-converge} but with $\alpha_j$ and $\beta_j$. As a consequence of the duality result (i.e., Theorem \ref{thm:duality}), this means that the one-cut scheme \eqref{eq:one-cut} can use also $\tau_j$ different from $j/(j+2)$. It is worth noting that these new choices of $\tau_j$ and their corresponding convergence proofs are only made possible by the duality connection discovered in this section.

The following theorem is a counterpart of Theorem \ref{thm:CG-converge} in the case of choosing $\tau_j$ of CG as in \eqref{eq:alphaj} or \eqref{eq:betaj}.
An implicit assumption is that we are solving \eqref{eq:dual} as the dual to the proximal subproblem \eqref{eq:phi-lam} within PDPB. Consequently, the iteration count $k$ in PDPB satisfies $k \le 2d_0^2/(\lam \bar \varepsilon)$, in accordance with the assumption in Lemma \ref{lem:t1}.

\begin{theorem}
Consider Algorithm \ref{alg:CG} with $\tau_j$ as in \eqref{eq:alphaj} or \eqref{eq:betaj}, then for every $j \ge 1$, \eqref{ineq:CG-convergence} holds
where $u_j$ is as in \eqref{def:uj} with $\tau_j = j/(j+2)$ and $z_j$ is as in \eqref{eq:z+} with $\tau_j$ as in \eqref{eq:alphaj} or \eqref{eq:betaj} correspondingly.
\end{theorem}

\begin{proof}
First, it follows from Lemma \ref{lem:FW-basic} and the definition of $z_{j+1}$ in \eqref{eq:z+} that for any $\tau_j \in [0,1]$,
\begin{equation}\label{ineq:tauj}
  \psi(\tau_j z_j + (1-\tau_j) \bar z_j) \stackrel{\eqref{ineq:CG-fundamental}}\le \psi(z_j) - (1-\tau_j) S(z_j) + \frac{(1-\tau_j)^2 \lam}{2}\| \bar z_j - z_j\|^2.
\end{equation}
Claim: In either case of Algorithm \ref{alg:CG} with $\tau_j$ as in \eqref{eq:alphaj} or \eqref{eq:betaj}, we have for any $\tau_j \in [0,1]$,
\begin{equation}\label{ineq:claim}
    \psi(z_{j+1}) \le \psi(z_j) - (1-\tau_j) S(z_j) + \frac{(1-\tau_j)^2 \lam}{2}\| \bar z_j - z_j\|^2.
\end{equation}
In the case of $\alpha_j$ in \eqref{eq:alphaj}, it is easy to see from \eqref{eq:z+} that $z_{j+1} = \alpha_j z_j + (1-\alpha_j) \bar z_j$,
which together with \eqref{ineq:tauj} with $\tau_j=\alpha_j$ implies that
\begin{equation}\label{ineq:alpha}
    \psi(z_{j+1}) \le \psi(z_j) - (1-\alpha_j) S(z_j) + \frac{(1-\alpha_j)^2 \lam}{2}\| \bar z_j - z_j\|^2.
\end{equation}
Noting from \eqref{eq:alphaj} that
\[
1-\alpha_j =\min \left\{1, \frac{S(z_j)}{\lam \|z_j-\bar z_j\|^2}\right\},
\]
which minimizes the right-hand side of \eqref{ineq:claim} as a quadratic function of $1-\tau_j$ over the interval [0,1]. Hence, \eqref{ineq:alpha} immediately implies that \eqref{ineq:claim} holds for any $\tau_j \in [0,1]$.
In the case of $\beta_j$ in \eqref{eq:betaj}, it is clear that for any $\tau_j \in [0,1]$,
\[
\psi(z_{j+1}) \stackrel{\eqref{eq:z+}}= \psi(\beta_j z_j + (1-\beta_j) \bar z_j) \stackrel{\eqref{eq:betaj}}\le \psi(\tau_j z_j + (1-\tau_j) \bar z_j).
\]
Hence, \eqref{ineq:claim} immediately follows from this observation and \eqref{ineq:tauj}. We have thus proved the claim.
Except for $z_{j+1}$ in \eqref{ineq:claim} is computed as in \eqref{eq:z+} with $\tau_j$ replaced by $\alpha_j$ or $\beta_j$, the claim is the same as Lemma \ref{lem:FW-basic}.
Finally, the conclusion of the theorem holds as a consequence of Theorem \ref{thm:CG-converge}.
\end{proof}

\subsection{New variants of CG inspired by GBM implementations} \label{subsec: new CG}

Motivated by possible $\tau_j$'s used in CG, we develop in Subsection \ref{subsec:new GBM} new implementations of GBM, i.e., the one-cut scheme \eqref{eq:one-cut} with $\alpha_j$ and $\beta_j$ in \eqref{eq:alphaj} and \eqref{eq:betaj}, respectively.
In this subsection, we further exploit the duality between PDCP and CG from the other direction by developing novel CG variants with inspiration from other GBM implementations used in PDCP.

Apart from the one-cut scheme \eqref{eq:one-cut}, Subsection~3.1 of \cite{liang2024unified} also provides two other candidates for GBM, i.e., two-cuts and multiple-cuts schemes, which are standard cut-aggregation and cutting-plane models, respectively.

To begin with, we first briefly review the two-cuts scheme.
It starts from $\Gamma_1(\cdot) = \bar \Gamma_0(\cdot) = \ell_f(\cdot;x_0)$. For $j\ge 1$, given 
\begin{equation}\label{eq:tc}
    \Gamma_j(\cdot) = \max \{\bar \Gamma_{j-1}(\cdot), \ell_f(\cdot; x_{j-1})\}
\end{equation}
where $\bar \Gamma_{j-1}$ is an affine function, the two-cuts scheme recursively updates $\Gamma_{j+1}$ as in \eqref{eq:tc}, i.e., $\Gamma_{j+1}(\cdot) = \max \left\{\bar \Gamma_j(\cdot), \ell_f(\cdot; x_j)\right\}$,
which always maintains two cuts.
The auxiliary bundle model $\bar \Gamma_j$ is updated as
\begin{equation}\label{eq:tc-bar}
    \bar \Gamma_j(\cdot) = \theta_{j-1} \bar \Gamma_{j-1}(\cdot) + (1-\theta_{j-1})\ell_f(\cdot;x_{j-1}), 
\end{equation}
where $\theta_{j-1}$ is the Lagrange multiplier associated with the first constraint in the problem below
\begin{equation}\label{eq:TC-min}
    \underset{(u,r) \in \R^n \times \R}\min\left\{r+h^\lam(u): \bar \Gamma_{j-1}(u) \le r, \, \ell_f(u,x_{j-1}) \le r \right\}.
\end{equation}
Proposition D.1 in \cite{liang2024unified} shows that the above two-cuts scheme satisfies GBM.

Recall the previous options of $\tau_j$ in CG (see \eqref{eq:z+}), i.e., $j/(j+2)$, \eqref{eq:alphaj}, and \eqref{eq:betaj}, are all determined once we know $z_j$ and $\bar z_j$. One natural way to generalize CG is to leave $\tau_j$ and, consequently, $z_{j+1}$ undetermined, deferring their computation to the subsequent iteration. Therefore, \eqref{eq:bar-zj} and \eqref{eq:z+} are insufficient to determine \(\tau_j\) and \(z_{j+1}\), and more conditions are needed.
For instance, motivated by the two-cuts scheme above, we additionally require
\begin{equation}\label{eq:require}
    x_j = \nabla (h ^\lam)^* (-z_j), \quad \theta_{j-1} \bar \Gamma_{j-1}(x_j) + (1-\theta_{j-1})\ell_f(x_j;x_{j-1}) =  \Gamma_j(x_j),
\end{equation}
where $z_j = \theta_{j-1} z_{j-1} + (1-\theta_{j-1}) \bar z_{j-1}$ following from \eqref{eq:z+}.
Note that \eqref{eq:TC-min} is equivalent to \eqref{eq:xj} with $\Gamma_j$ as in \eqref{eq:tc}, and hence the optimal solution to \eqref{eq:TC-min} is $(x_j,\Gamma_j(x_j))$. As a result, with the understanding that $z_j = \nabla \bar \Gamma_j$ and $\bar z_j = f'(x_j)$, the first identity in \eqref{eq:require} corresponds to the optimality of \eqref{eq:TC-min}, and the second one in \eqref{eq:require} corresponds to the complementary slackness of \eqref{eq:TC-min}. Moreover, it follows from \eqref{eq:mc} that $\partial \Gamma_j$ is the convex hull of $\nabla \bar \Gamma_{j-1}$ and $f'(x_{j-1})$, and hence that
\[
z_j = \nabla \bar \Gamma_j = \theta_{j-1} \nabla \bar \Gamma_{j-1} + (1-\theta_{j-1})f'(x_{j-1}) \in \partial \Gamma_j(x_j).
\]
The discussion above verifies that Theorem \ref{thm:duality} also holds in the context of the two-cuts scheme. In other words, in the spirit of Theorem \ref{thm:duality}, this new CG variant is the dual method of PDCP with the two-cuts implementation of GBM.

We now turn to review the multi-cuts scheme and discuss its implication in generalizing CG. For $j\ge 1$, given an index set $I_j \subseteq \{0, \cdots, j-1\}$, the multi-cuts scheme sets
\begin{equation}\label{eq:mc}
    \Gamma_j(\cdot) = \max \left\{\ell_f(\cdot ; x_i): i \in I_j\right\}.
\end{equation}
The index set $I_j$ starts from $I_1 = \{0\}$ and recursively updates as
\[
I_{j+1} = \bar I_{j+1} \cup \{j\}, \quad \bar I_{j+1} = \{i \in I_j: \theta_j^i>0\},
\]
where $\theta_j^i$ is the Lagrange multiplier associated with the constraint $\ell_f(u;x_i) \le r$ in the problem below
\begin{equation}\label{eq:MC-min}
    \underset{(u,r) \in \R^n \times \R}\min\left\{r + h^\lam(u): \ell_f(u;x_i) \le r, \, \forall i\in I_j\right\}.
\end{equation}
Here, $\bar \Gamma_j(\cdot) = \max \left\{\ell_f(\cdot ; x_i): i \in \bar I_j\right\}$. Proposition D.2 in \cite{liang2024unified} shows that the above multi-cuts scheme satisfies GBM.

The recursion \eqref{eq:z+} indicates that $z_j$ in CG is a convex combination of $\{z_1, \bar z_1, \ldots, \bar z_{j-1}\}$. Hence, a more general candidate of $z_j$ is any point in the convex hull of $\{z_1, \bar z_1, \ldots, \bar z_{j-1}\}$.
Similar to the discussion of the new CG motivated by the two-cuts scheme, we also need to introduce conditions to determine $z_j$ in this generalization. For instance, inspired by the multi-cuts scheme above, we specifically compute
\begin{equation}\label{eq:zjtheta}
    z_j = \sum_{i\in I_j} \theta_j^i \bar z_i
\end{equation}
with the convention that $\bar z_0 = z_1$, where $\theta_j^i$ is the corresponding Lagrange multiplier for \eqref{eq:MC-min}.
Now, the positive multiplier $\theta_j^i$ (primal perspective) also serves as the convex combination parameter (dual perspective).
Note that \eqref{eq:MC-min} is equivalent to \eqref{eq:xj} with $\Gamma_j$ as in \eqref{eq:mc}, and hence the optimal solution to \eqref{eq:MC-min} is $(x_j,\Gamma_j(x_j))$. Again, it is easy to verify that
\[
z_j \in \partial \Gamma_j(x_j), \quad x_j = \nabla ( h ^\lam)^* (-z_j), \quad f'(x_j) = \bar z_j,
\]
and hence Theorem \ref{thm:duality} holds in the context of the multi-cuts scheme.
In other words, following the spirit of Theorem \ref{thm:duality}, this generalization of CG serves as the dual method of PDCP, implemented with the multi-cuts scheme.

Since the number of nonzero $\theta_j^i$ could be small (compared to $j$), $z_j$ has a sparse representation in terms of $\{\bar z_0, \bar z_1, \ldots, \bar z_{j-1}\}$. Assuming $\{\bar z_j\}$ is a sequence of sparse vectors, then $z_j$ is sparse, and indeed sparser than those generated by CG using \eqref{eq:z+} with $\tau_j$ being $j/(j+2)$, $\alpha_j$, $\beta_j$, and $\theta_j$. 

Leveraging the primal-dual connections between PDCP with two-cuts and multi-cuts schemes and the novel CG variants, we present the following convergence result for the latter. The proof is omitted, as it directly follows from Proposition \ref{prop:inner} and Lemma~\ref{lem:t1}, which establish the convergence of PDCP under the two-cuts and multi-cuts schemes.
An implicit assumption is that we are solving \eqref{eq:dual} as the dual to the proximal subproblem \eqref{eq:phi-lam} within PDPB. Consequently, the iteration count $k$ in PDPB satisfies $k \le 2d_0^2/(\lam \bar \varepsilon)$, in accordance with the assumption in Lemma \ref{lem:t1}.

\begin{theorem}
 Consider the two CG variants described in this subsection, then $z_j$ generated in each variant satisfies
 \[
 \phi^\lam(\tx_j) + \psi(z_j) \le \frac{8M (3d_0 + \lam M)}{j(j+1)} + \frac{16\lam M^2}{j+1},
 \]
 where $\tilde x_j$ is as in \eqref{ineq:txj} with $\tau_j = j/(j+2)$.
\end{theorem}

\section{Proximal bundle method for SPP}\label{sec:PB-SPP}

In this section, we consider the convex-concave nonsmooth composite SPP \eqref{eq:spp}.
More specifically, we assume the following conditions hold:
\begin{itemize}
    \item[(B1)] a subgradient oracle $f_x': \dom h_1 \times \dom h_2 \to \R^n$ and a supergradient oracle $f_y': \dom h_1 \times \dom h_2 \to \R^m$ are available, that is, we have $f_x'(u,v) \in \partial_x f(u,v)$ and $f_y'(u,v) \in \partial_y f(u,v)$ for every $(u,v)\in \dom h_1 \times \dom h_2$;
    \item[(B2)] both $f_x'$ and $f_y'$ are uniformly bounded by some positive scalar $M$ over $\dom h_1$ and $\dom h_2$, i.e., for every pair $(u,v)\in \dom h_1 \times \dom h_2$,
    \begin{equation}\label{ineq:spp-Lips}
        \|f_x'(u,v)\| \le M, \quad \|f_y'(u,v)\| \le M;
    \end{equation}
    \item[(B3)] $\dom h_1 \times \dom h_2$ is bounded with finite diameter $D>0$; 
    \item[(B4)] the proximal mappings of $h_1$ and $h_2$ are easy to compute;
    \item[(B5)] the set of saddle points of problem \eqref{eq:spp} is nonempty.
\end{itemize}

Given a pair $(x,y) \in \dom h_1 \times \dom h_2$, for every $(u,v)\in \dom h_1 \times \dom h_2$, define
\[
\ell_{f(\cdot,y)}(u;x)  = f(x,y) + \inner{f_x'(x,y)}{u-x}, \quad \ell_{f(x,\cdot)}(v;y)  = f(x,y) + \inner{f_y'(x,y)}{v-y}.
\]
It is easy to see from (B2) that for fixed $(x,y)$ and every $(u,v)\in \dom h_1 \times \dom h_2$,
\begin{equation}\label{ineq:Lips-spp}
    f(u,y)-\ell_{f(\cdot,y)}(u;x) \le 2 M\|u-x\|, \quad \ell_{f(x,\cdot)}(v;y) - f(x,v) \le 2M\|v-y\|.
\end{equation}

We say a pair $(x_*, y_*) \in \dom h_1 \times \dom h_2$ is a saddle-point of \eqref{eq:spp} if for every pair $(u,v)\in \dom h_1 \times \dom h_2$,
\begin{equation}\label{def:spp}
    \phi(x_*,v) \le \phi(x_*,y_*) \le \phi(u,y_*).
\end{equation}
We say a pair $(x, y) \in \dom h_1 \times \dom h_2$ is a $\bar \varepsilon$-saddle-point of \eqref{eq:spp} if
\begin{equation}\label{def:eps-spp}
    0 \in \partial_{\bar \varepsilon} [\phi(\cdot, y)-\phi(x, \cdot)](x,y).
\end{equation}
It is well-known that SPP \eqref{eq:spp} is equivalent to
\begin{equation}\label{eq:spp-Phi}
    \min_{x\in\R^n, y\in \R^m} \{\Phi(x,y) := \varphi(x) - \psi(y)\},
\end{equation}
where 
\begin{equation}\label{def:spp-func}
    \varphi(x)=\max _{y \in \R^m} \phi(x, y), \quad \psi(y)=\min _{x \in \R^n} \phi(x, y).
\end{equation}
As a consequence, an equivalent definition of $\bar \varepsilon$-saddle-point is as follows: a pair $(x, y) \in \dom h_1 \times \dom h_2$ satisfying
\begin{equation}\label{def:eps-spp-equiv}
    \Phi(x,y) = \varphi(x)-\psi(y) \le \bar \varepsilon.
\end{equation}
The equivalence between \eqref{def:eps-spp} and \eqref{def:eps-spp-equiv} is given in Lemma \ref{lem:spp-equiv}.
Another related but weaker notion is a pair $(x, y) \in \dom h_1 \times \dom h_2$ satisfying
\begin{equation}\label{def:spp-weak}
    -\bar \varepsilon\le \phi(x,y) - \phi(x_*,y_*) \le \bar \varepsilon.
\end{equation}
The implication from \eqref{def:eps-spp} to \eqref{def:spp-weak} is given in Lemma \ref{lem:101}.

The composite subgradient method for SPP \eqref{eq:spp} denoted by CS-SPP$(x_0,y_0,\lam)$, where $(x_0,y_0)\in \dom h_1 \times \dom h_2$ is the initial pair and $\lam>0$ is the prox stepsize, recursively computes
\begin{align}
    x_k &= \underset{u\in \R^n}\argmin\left\{\ell_{f(\cdot,y_{k-1})}(u;x_{k-1}) + h_1(u) + \frac{1}{2\lam}\|u-x_{k-1}\|^2\right\}, \label{eq:spp-xk} \\
    y_k &= \underset{v\in \R^m}\argmin\left\{-\ell_{f(x_{k-1},\cdot)}(v;y_{k-1}) + h_2(v) + \frac{1}{2\lam}\|v-y_{k-1}\|^2\right\} \label{eq:spp-yk}.
\end{align}

For given tolerance $\bar \varepsilon>0$, letting $\lam = \bar \varepsilon/(32M^2)$, then the iteration-complexity for CS-SPP$(x_0,y_0,\lam)$ to generate a $\bar \varepsilon$-saddle point of \eqref{eq:spp} is bounded by ${\cal O}(M^2 D^2/\bar \varepsilon^2)$ (see Theorem~\ref{thm:cs-spp}).

\subsection{Inexact proximal point framework for SPP} \label{subsec:IPPF}

The generic PPM for solving \eqref{eq:spp-Phi} iteratively solves the proximal subproblem
\begin{equation}\label{eq:ppm-spp}
    (x_k, y_k) = \underset{x\in\R^n, y\in \R^m}\argmin\left\{ \Phi(x,y) + \frac{1}{2\lam_k} \|x-x_{k-1}\|^2 + \frac{1}{2\lam_k} \|y-y_{k-1}\|^2 \right\},
\end{equation}
which motivates the following proximal point formulation for solving \eqref{eq:spp}
\begin{equation}\label{eq:ppm-spp-1}
    (x_k, y_k) = \underset{x\in\R^n}\argmin \underset{y\in\R^m}\argmax\left\{\phi(x,y) + \frac1{2\lam_k}\|x-x_{k-1}\|^2 - \frac{1}{2\lam_k}\|y-y_{k-1}\|^2\right\}.
\end{equation}
However, both \eqref{eq:ppm-spp} and \eqref{eq:ppm-spp-1} are only conceptual PPMs for SPP. In this subsection, we introduce the generic IPPF for solving SPP \eqref{eq:spp} and show that CS-SPP described previously is a concrete example of IPPF.

\begin{algorithm}[H]
\begin{algorithmic}
\caption{Inexact Proximal Point Framework for SPP \eqref{eq:spp}}
\REQUIRE given initial pair $(x_0, y_0) \in \dom h_1 \times \dom h_2$ and scalar $\sigma \in [0,1]$
\FOR{$k = 1, 2, \cdots$}
\STATE $\bullet$ choose $\lam_k >0$, $\varepsilon_k>0$, and $\delta_k>0$ and find $(x_k, y_k)\in \dom h_1 \times \dom h_2$ and $(\tx_k, \ty_k)\in \dom h_1 \times \dom h_2$ such that
\begin{equation}\label{incl}
   \left( \frac{x_{k-1}-x_k}{\lam_k}, \frac{y_{k-1}-y_k}{\lam_k} \right)
 \in \partial_{\varepsilon_k} [\phi(\cdot,y_{k-1}) - \phi(x_{k-1},\cdot)](\tx_k,\ty_k)
\end{equation}
and
\begin{equation}\label{ineq:IPP}
		\left\|x_k-\tx_k\right\|^2+\left\|y_k-\ty_k\right\|^2+2 \lam_k \varepsilon_k \le \delta_k + \sigma \left(\left\|\tx_k-x_{k-1}\right\|^2+\left\|\ty_k-y_{k-1}\right\|^2\right).
	\end{equation}
\ENDFOR
\end{algorithmic} 
\end{algorithm}

\begin{lemma}\label{lem:pd-equiv}
    For every $k\ge 1$, define $p_k:\R^n\to \R$ and $d_k:\R^m\to \R$ as follows
    \begin{equation}\label{def:p-d-k}
        p_k(\cdot) := f(\cdot, y_{k-1})+ h_1(\cdot), \quad d_k(\cdot) := -f(x_{k-1},\cdot)+ h_2(\cdot).
    \end{equation}
    Then, the inclusion \eqref{incl} is equivalent to for every $(u,v) \in \dom h_1 \times \dom h_2$,
    \begin{align}\label{ineq:target}
        & p_k(u) + d_k(v) -p_k(\tx_k) -d_k(\ty_k) \nn\\
        \ge &\frac{1}{\lam_k}\inner{x_{k-1}-x_k}{u-\tx_k} + \frac{1}{\lam_k}\inner{y_{k-1}-y_k}{v-\ty_k} - \varepsilon_k.
    \end{align}
\end{lemma}

\begin{proof}
It follows from \eqref{incl} and the definition of $\varepsilon$-subdifferential \eqref{def:subdiff} that for every pair $(u,v) \in \dom h_1 \times \dom h_2$,
\begin{align*}
    &\phi(u, y_{k-1})-\phi(x_{k-1}, v) - [\phi(\tx_k, y_{k-1})-\phi(x_{k-1}, \ty_k)] \\
    \ge & \frac{1}{\lam_k}\inner{x_{k-1}-x_k}{u-\tx_k} + \frac{1}{\lam_k}\inner{y_{k-1}-y_k}{v-\ty_k} - \varepsilon_k.
\end{align*}
Observing from the definitions of $p_k$ and $d_k$ in \eqref{def:p-d-k} that
\[
p_k(u) + d_k(v) -p_k(\tx_k) -d_k(\ty_k) = \phi(u, y_{k-1})-\phi(x_{k-1}, v) - [\phi(\tx_k, y_{k-1})-\phi(x_{k-1}, \ty_k)],
\]
which together with the above inequality implies that \eqref{ineq:target} holds.
\end{proof}

We are now ready to present the result showing that CS-SPP is an instance of IPPF with certain parameterizations. The proof is postponed to Subsection \ref{subsec:pf-CS-IPPF}.

\begin{proposition}\label{prop:CS-SPP-IPPF}
    Given $(x_0,y_0)\in \dom h_1 \times \dom h_2$, $\delta>0$, and $\lam=\sqrt{\delta/8M^2}$, then
     CS-SPP$(x_0,y_0,\lam)$ is an instance of IPPF with $\sigma=1$, $(\lam_k,\delta_k)=(\lam,\delta)$ for every $k\ge 1$, $(\tx_k,\ty_k) = (x_k,y_k)$ where $x_k$ and $y_k$ are as in \eqref{eq:spp-xk} and \eqref{eq:spp-yk}, respectively, and $\varepsilon_k = \varepsilon_k^x + \varepsilon_k^y$ where
     \begin{align}
         \varepsilon_k^x &= f(x_k,y_{k-1}) - \ell_{f(\cdot,y_{k-1})}(x_k;x_{k-1}), \label{def:ekx}\\
         \varepsilon_k^y &= -f(x_{k-1},y_k) + \ell_{f(x_{k-1},\cdot)}(y_k;y_{k-1}). \label{def:eky}
     \end{align}
\end{proposition}

\subsection{Proximal bundle method for SPP} \label{subsec:PB-SPP}

In this subsection, we describe another instance of IPPF, namely PB-SPP, for solving SPP \eqref{eq:spp}. The inclusion of PB-SPP as an instance of IPPF is presented in Proposition~\ref{prop:CS-SPP-IPPF} below.

We start by stating PB-SPP. 

\begin{algorithm}[H]
\caption{Proximal Bundle for SPP \eqref{eq:spp}, PB-SPP$(x_0,y_0,\bar \varepsilon)$}\label{alg:PB-SPP}
\begin{algorithmic}
\REQUIRE given $(x_0,y_0) \in \dom h_1 \times \dom h_2$ and $\bar \varepsilon>0$
\FOR{$k=1,2,\cdots$}
\STATE $\bullet$  call oracles $(x_k,\tx_k) = \text{PDCP}(x_{k-1},\lam_k, \bar \varepsilon/4)$ and $(y_k,\ty_k) = \text{PDCP}(y_{k-1},\lam_k, \bar \varepsilon/4)$ and compute 
\begin{equation}\label{def:hxy}
    \bar x_k = \frac{1}{k}\sum_{i=1}^{k} \tx_i , \quad \bar y_k = \frac{1}{k}\sum_{i=1}^{k} \ty_i.
\end{equation}
\ENDFOR
\end{algorithmic} 
\end{algorithm}

Inspired by PPM \eqref{eq:ppm-spp-1} for solving SPP \eqref{eq:spp}, the $k$-th iteration of PB-SPP aims at approximately solving the decoupled proximal subproblems, i.e.,
\begin{align}
    &\min_{x\in \R^n}\left\{f(x,y_{k-1}) + h_1(x) + \frac{1}{2\lam_k}\|u-x_{k-1}\|^2\right\}, \label{eq:spp-dcp-x}\\
    &\min_{y\in \R^m}\left\{ -f(x_{k-1},y) + h_2(y) + \frac{1}{2\lam_k}\|v-y_{k-1}\|^2\right\}. \label{eq:spp-dcp-y}
\end{align}
Hence, the underlying $f$ in the call to PDCP$(x_{k-1},\lam_k, \bar \varepsilon)$ is $f(\cdot,y_{k-1})$ and the underlying $f$ in the call to PDCP$(y_{k-1},\lam_k, \bar \varepsilon)$ is $-f(x_{k-1},\cdot)$.
Correspondingly, similar to \eqref{eq:xk}, by calling the subroutine PDCP, PB-SPP exactly solves
\begin{align}
    x_k &= \underset{u\in\R^n}\argmin\left\{\Gamma_k^x(u) + h_1(u) + \frac{1}{2\lam_k}\|u-x_{k-1}\|^2 \right\}, \label{eq:xk-spp} \\
    y_k &= \underset{v\in\R^m}\argmin\left\{-\Gamma_k^y(v) + h_2(v) + \frac{1}{2\lam_k}\|v-y_{k-1}\|^2 \right\}, \label{eq:yk-spp}
\end{align}
where $\Gamma_k^x(\cdot)$ and $-\Gamma_k^y(\cdot)$ are the cutting-plane models constructed for $f(\cdot,y_{k-1})$ and $-f(x_{k-1},\cdot)$, respectively, by GBM (see step 2 of Algorithm \ref{alg:PDCP}). Hence, by the construction in GBM (i.e., Algorithm \ref{alg:GBM}) and the convexity of $f(\cdot,y_{k-1})$ and $-f(x_{k-1},\cdot)$, we have
\begin{equation}\label{ineq:cvx}
    \Gamma_k^x(\cdot) \le f(\cdot,y_{k-1}), \quad -\Gamma_k^y(\cdot) \le -f(x_{k-1},\cdot).
\end{equation}
Since GBM is a generic scheme, the models $\Gamma_k^x(\cdot)$ and $-\Gamma_k^y(\cdot)$ can be any one satisfying GBM, e.g., one-cut, two-cuts, and multiple-cuts schemes (i.e., (E1)-(E3)) described in Subsection~3.1 of \cite{liang2024unified}. As a result, PB-SPP is a template for many possible methods using GBM as their bundle management.

For ease of the convergence analysis of PB-SPP, we define
\begin{equation}\label{def:p-d}
    p_k^\lam (\cdot) : = p_k (\cdot) + \frac{1}{2\lam_k}\|\cdot - x_{k-1}\|^2, \quad d_k^\lam (\cdot) : = d_k (\cdot)+ \frac{1}{2\lam_k}\|\cdot - y_{k-1}\|^2,
\end{equation}
where $p_k$ and $d_k$ are as in \eqref{def:p-d-k}, $m_k^x$ and $m_k^y$ as the optimal values of \eqref{eq:xk-spp} and \eqref{eq:yk-spp}, respectively, and
\begin{equation}\label{def:tj-xy}
    t_k^x = p_k^\lam(\tx_k) - m_k^x, \quad t_k^y = d_k^\lam(\ty_k) - m_k^x.
\end{equation}

Following from Proposition \ref{prop:inner} and a simplification of Lemma \ref{lem:t1} using (B3), we obtain the convergence rates of $t_k^x$ and $t_k^y$.
We omit the proof since it is almost identical to that of Proposition~\ref{prop:inner}.

\begin{proposition}\label{prop:spp-inner}
    Considering Algorithm \ref{alg:PDCP} with $\tau_j = j/(j+2)$, then for every $j_k \ge 1$, we have
\[
t_k^x\le \frac{4MD}{l_k(l_k+1)} + \frac{16\lam_k M^2}{l_k+1}, \quad t_k^y\le \frac{4MD}{l_k(l_k+1)} + \frac{16\lam_k M^2}{l_k+1},
\]
where $l_k$ denotes the length of the $k$-th cycle ${\cal C}_k$ (i.e., $l_k = |{\cal C}_k| = j_k-i_k+1$).
\end{proposition}

Given Proposition \ref{prop:spp-inner}, PDCP is able to solve \eqref{eq:spp-dcp-x} and \eqref{eq:spp-dcp-y} to any desired accuracy. For given tolerance $\bar \varepsilon>0$, the calls to PDCP in Algorithm \ref{alg:PB-SPP} guarantees
\begin{equation}\label{ineq:spp-accuracy}
    t_k^x \le \frac{\bar \varepsilon}4, \quad
    t_k^y \le \frac{\bar \varepsilon}4.
\end{equation}
Starting from \eqref{ineq:spp-accuracy}, we establish the iteration-complexity for PB-SPP to find a $\bar \varepsilon$-saddle-point of SPP \eqref{eq:spp}.

\begin{lemma}\label{lem:spp-outer}
For every $k \ge 1$ and  $(u,v) \in \dom h_1 \times \dom h_2$, we have
    \begin{align}
        p_k(\tx_k) &- p_k(u) \le \frac{\bar \varepsilon}4  + \frac{1}{2\lam_k}\|u- x_{k-1}\|^2 - \frac{1}{2\lam_k}\|u-x_k\|^2 - \frac{1}{2\lam_k}\|\tx_k - x_{k-1}\|^2, \label{ineq:p} \\
        d_k(\ty_k) &- d_k(v) \le \frac{\bar \varepsilon}4  + \frac{1}{2\lam_k}\|v- y_{k-1}\|^2 - \frac{1}{2\lam_k}\|v - y_k\|^2 - \frac{1}{2\lam_k}\|\ty_k - y_{k-1}\|^2. \label{ineq:d}
    \end{align}
\end{lemma}

\begin{proof}
We only prove \eqref{ineq:p} to avoid duplication. Inequality \eqref{ineq:d} follows similarly. Noting that the objective in \eqref{eq:xk-spp} is $\lam_k^{-1}$-strongly convex and using the definition of $m_k^x$, we have for every $u \in \R^n$,
\[
\Gamma_k^x(u) + h_1(u) + \frac1{2\lam_k}\|u-x_{k-1}\|^2 \ge m_k^x + \frac{1}{2\lam_k} \|u-x_{k}\|^2. 
\]
It follows from the definition of $p_k$ in \eqref{def:p-d-k} and the first inequality in \eqref{ineq:cvx} that $p_k(\cdot) \ge (\Gamma_k^x + h_1)(\cdot)$. Hence, we have for every $u \in \R^n$,
\[
p_k^\lam(\tx_k) - p_k(u) \le p_k^\lam(\tx_k) - m_k^x + \frac{1}{2\lam_k}\|u-x_{k-1}\|^2 - \frac{1}{2\lam_k}\|u - x_k\|^2.
\]
Therefore, inequality \eqref{ineq:p} immediately follows from the definition of $t_k^x$ in \eqref{def:tj-xy} and the first inequality in \eqref{ineq:spp-accuracy}.
\end{proof}

\begin{lemma}\label{lem:spp-outer-comb}
For every $k \ge 1$ and  $(u,v) \in \dom h_1 \times \dom h_2$, we have
    \begin{equation}\label{ineq:spp-outer-comb}
        \phi(\tx_k,v) - \phi(u,\ty_k) \le \frac{\bar \varepsilon}2 + \frac{1}{2\lam_k}\|z_{k-1}-w\|^2 - \frac{1}{2\lam_k}\|z_{k} - w\|^2 + 4\lam_kM^2,
    \end{equation}
    where $w = (u,v)$ and $z_k = (x_k,y_k)$.
\end{lemma}

\begin{proof}
It follows from (B2) that for every $u \in \dom h_1$,
\[ 
f(u,y_{k-1}) - f(u,\ty_k) \stackrel{\eqref{ineq:spp-Lips}}\le M \|\ty_k - y_{k-1}\|, \quad f(\tx_k,\ty_k) - f(\tx_k,y_{k-1}) \stackrel{\eqref{ineq:spp-Lips}}\le M \|\ty_k - y_{k-1}\|.
\]
Noting from \eqref{def:p-d-k} that $p_k(\tilde x_k) - p_k(u) = f(\tx_k, y_{k-1}) + h_1(\tx_k) - f(u,y_{k-1})- h_1(u)$, using this relation and the above inequality in \eqref{ineq:p}, we have for every $u \in \dom h_1$,
\begin{align*}
& f(\tx_k,\ty_k) + h_1(\tx_k) - f(u,\ty_k)- h_1(u) \\
\stackrel{\eqref{ineq:p}}\le & \frac{\bar \varepsilon}4  + \frac{1}{2\lam_k}\|x_{k-1}-u\|^2 - \frac{1}{2\lam_k}\|x_{k} - u\|^2 + 2M\|\ty_k-y_{k-1}\| - \frac{1}{2\lam_k}\|\tx_k- x_{k-1}\|^2.
\end{align*}
Similarly, using \eqref{ineq:d}, we can prove for every $v \in \dom h_2$,
\begin{align*}
    &-f(\tx_k,\ty_k)+ h_2(\ty_k)  + f(\tx_k,v)- h_2(v) \\
    \stackrel{\eqref{ineq:d}}\le & \frac{\bar \varepsilon}4  + \frac{1}{2\lam_k}\|y_{k-1}-v\|^2 - \frac{1}{2\lam_k}\|y_{k} - v\|^2 + 2M\|\tx_k-x_{k-1}\| - \frac{1}{2\lam_k}\|\ty_k- y_{k-1}\|^2.
\end{align*}
Noting that $2M a - a^2/(2\lam_k) \le 2\lam_k M^2$ for $a\in\R$ and summing the above two inequalities, we obtain
\begin{align*}
    \phi(\tx_k,v) - \phi(u,\ty_k)
    \stackrel{\eqref{eq:spp}}=&f(\tx_k,v)+ h_1(\tx_k) - h_2(v) - f(u,\ty_k)- h_1(u) + h_2(\ty_k) \\
    \le & \frac{\bar \varepsilon}2 + \frac{1}{2\lam_k}\|z_{k-1}-w\|^2 - \frac{1}{2\lam_k}\|z_{k} - w\|^2 + 4\lam_kM^2,
\end{align*}
where the identity is due to the definition of $\phi(\cdot,\cdot)$ in \eqref{eq:spp}.
\end{proof}

\begin{proposition}\label{prop:pb-spp}
    For every $k \ge 1$, setting $\lam_k = \lam_1/\sqrt{k}$ for some $\lam_1>0$, then for every $(u,v) \in \dom h_1 \times \dom h_2$, we have
    \begin{equation}\label{ineq:spp-gap}
        \varphi(\bar x_k) - \psi(\bar y_k) \le \frac{\bar \varepsilon}2 + \frac{8\lam_1 M^2}{\sqrt{k}} + \frac{D^2}{2\lam_1 \sqrt{k}},
    \end{equation}
    where $\bar x_k$ and $\bar y_k$ are as in \eqref{def:hxy}.
\end{proposition}

\begin{proof}
Replacing the index $k$ in \eqref{ineq:spp-outer-comb} by $i$, summing the resulting inequality from $i=1$ to $k$, and using \eqref{def:hxy} and the convexity of $\phi(\cdot,y)$ and $-\phi(x,\cdot)$, we have for every $(u,v) \in \dom h_1 \times \dom h_2$,
\begin{equation}\label{ineq:spp-sum}
    \phi(\bar x_k,v) - \phi(u,\bar y_k)
    \le \frac{\bar \varepsilon}2 + \frac1k \sum_{i=1}^k \left[ \frac{1}{2\lam_i} (\|z_{i-1}-w\|^2 -\|z_i - w\|^2)  + 4\lam_i M^2\right].
\end{equation}
It follows from the fact that $\lam_k = \lam_1/\sqrt{k}$ and assumption (B3) that
\begin{align}
\frac1k \sum_{i=1}^k \Big[ \frac{1}{2\lam_i} \Big(\|z_{i-1}-w\|^2 -&\|z_i - w\|^2\Big) \Big]  \le \frac1{2k} \left[\frac{\|z_0 - w\|^2}{\lam_1} + \sum_{i=1}^{k-1}\|z_i - w\|^2\left(\frac{1}{\lam_{i+1}}-\frac{1}{\lam_i}\right)\right]   \nn\\
& \le \frac{D^2}{2k} \left[\frac{1}{\lam_1} + \sum_{i=1}^{k-1}\left(\frac{1}{\lam_{i+1}}-\frac{1}{\lam_i}\right)\right] = \frac{D^2}{2 k \lam_k} = \frac{D^2}{2\lam_1\sqrt{k}}. \label{ineq:tele}
\end{align}
Observing that $\sum_{i=1}^k (1/\sqrt{i}) \le \int_0^k (1/\sqrt{x}) dx = 2 \sqrt{k}$, and hence
\[
\frac{1}{k}\sum_{i=1}^k 4\lam_iM^2 = \frac{1}{k}\sum_{i=1}^k \frac{4\lam_1M^2}{\sqrt{i}} \le \frac{8\lam_1 M^2}{\sqrt{k}}.
\]
This observation, \eqref{ineq:spp-sum}, and \eqref{ineq:tele} imply that
    \[
    \phi(\bar x_k,v) - \phi(u,\bar y_k) \le \frac{\bar \varepsilon}2 + \frac{8\lam_1 M^2}{\sqrt{k}} + \frac{D^2}{2\lam_1 \sqrt{k}}.
    \]
Maximizing the left-hand side over $(u,v) \in \R^n \times \R^m$ and using \eqref{def:spp-func} yield \eqref{ineq:spp-gap}.
\end{proof}

We are now ready to establish the iteration-complexity for PB-SPP to find a $\bar \varepsilon$-saddle-point.

\begin{theorem}\label{thm:pb-spp}
Given $(x_0,y_0,\bar \varepsilon) \in \dom h_1 \times \dom h_2 \times R_{++}$, letting $\lam_1 = D/(4M)$, then the iteration-complexity for PB-SPP$(x_0,y_0)$ to find a $\bar \varepsilon$-saddle-point $(\bar x_k, \bar y_k)$ of \eqref{eq:spp} is ${\cal O}((MD/\bar \varepsilon)^{2.5})$.
\end{theorem}

\begin{proof}
It follows from Proposition \ref{prop:pb-spp} with $\lam_1 = D/(4M)$ that
\[
\varphi(\bar x_k) - \psi(\bar y_k) \le \frac{\bar \varepsilon}2 + \frac{4 M D}{\sqrt{k}}.
\]
Hence, PB-SPP takes $k=64M^2D^2/\bar \varepsilon^2$ iterations to find the $\bar \varepsilon$-saddle-point $(\bar x_k, \bar y_k)$.
Using Proposition \ref{prop:spp-inner}, we know to have \eqref{ineq:spp-accuracy} holds for every cycle ${\cal C}_i$, it is sufficient to have
\[
l_i = \frac{\sqrt{32MD}}{\sqrt{\bar \varepsilon}} + \frac{128\lam_i M^2}{\bar \varepsilon} = \frac{\sqrt{32MD}}{\sqrt{\bar \varepsilon}} + \frac{32M D}{\bar \varepsilon \sqrt{i}},
\]
where the second identity is due to the facts that $\lam_i=\lam_1/\sqrt{i}$ and $\lam_1 = D/(4M)$.
As a consequence, the total number of iterations (of proximal mappings of $h_1$ and $h_2$, and of calls to subgradient oracles $f_x'$ and $f_y'$) is
\[
\sum_{i=1}^k l_i = \frac{\sqrt{32MD}}{\sqrt{\bar \varepsilon}} k + \sum_{i=1}^k \frac{32M D}{\bar \varepsilon \sqrt{i}} \le \frac{256\sqrt{2}M^{2.5} D^{2.5}}{\bar \varepsilon^{2.5}} + \frac{512 M^2 D^2}{\bar \varepsilon^2},
\]
where we use the facts that$\sum_{i=1}^k (1/\sqrt{i}) \le \int_0^k (1/\sqrt{x}) dx = 2 \sqrt{k}$ and $k=64M^2D^2/\bar \varepsilon^2$.
\end{proof}




Finally, we conclude this subsection by presenting that PB-SPP is an instance of IPPF. The proof is postponed to Subsection \ref{subsec:pf-PB-IPPF}.

\begin{proposition}\label{prop:PB-SPP-IPPF}
    Given $(x_0,y_0)\in \dom h_1 \times \dom h_2$, $\bar \varepsilon>0$, then
     PB-SPP$(x_0,y_0,\bar \varepsilon)$ is an instance of IPPF with $\sigma=0$, $\delta_k=\lam_k \bar \varepsilon/2$, and $\varepsilon_k = \varepsilon_k^x + \varepsilon_k^y$ where
     \begin{align}
         \varepsilon_k^x &= p_k(\tx_k) - (\Gamma_k^x+h_1)(x_k) + \frac{1}{\lam_k}\inner{x_{k-1}-x_k}{x_k-\tx_k}, \label{def:ekx-pb}\\
         \varepsilon_k^y &= d_k(\ty_k) -(-\Gamma_k^y+h_2)(y_k)+ \frac{1}{\lam_k}\inner{y_{k-1}-y_k}{y_k-\ty_k}. \label{def:eky-pb}
     \end{align}
\end{proposition}

\subsection{An optimal bound}\label{subsec:opt}

Note that the complexity bound ${\cal O}((MD/\bar \varepsilon)^{2.5})$ established in Theorem \ref{thm:pb-spp} holds for any bundle model $\Gamma_k^x$ and $-\Gamma_k^y$ generated by GBM, such as one-cut, two-cuts, and multiple-cuts schemes described in Subsection~3.1 of \cite{liang2024unified}.
However, the bound is worse than the optimal one ${\cal O}((MD/\bar \varepsilon)^2)$. This subsection is devoted to the development of the improved bound for the PB-SPP method whose subroutine PDCP uses the bundle model $\Gamma_j$ satisfying GBM but with \eqref{def:Gammaj} replaced by a stronger condition
\begin{equation}\label{ineq:Gamma1}
    \Gamma_{j+1}(\cdot) \ge \max\left\{\bar \Gamma_j(\cdot), \ell_f(\cdot; x_j)\right\}.
\end{equation} 
We also assume that the bundle model $\Gamma_j$ is $M$-Lipschitz continuous. It is easy to verify that both two-cuts and multiple-cuts schemes (i.e., \eqref{eq:tc} and \eqref{eq:mc}, respectively) satisfy the Lipschitz continuity and \eqref{ineq:Gamma1}.
However, the one-cut scheme \eqref{eq:one-cut} does not satisfy \eqref{ineq:Gamma1}.

The key to achieving the desired improvement lies in obtaining tighter bounds on $t_k^x$ and $t_k^y$ in Proposition~\ref{prop:spp-inner}. This, in turn, requires a more refined analysis of the PDCP subroutine used for solving \eqref{eq:spp-dcp-x} and \eqref{eq:spp-dcp-y}. To that end, we revisit the analysis of PDCP in Subsection~\ref{subsec:PDCP}, now under the setting where the condition \eqref{def:Gammaj} used in GBM is replaced by the stronger condition \eqref{ineq:Gamma1}.

To set the stage, we fix the prox center $x_{k-1}$ as in \eqref{eq:phi-lam}, denote it as $x_0$ to emphasize a local perspective within the current cycle, and recall the notation
\begin{equation}\label{eq:phi2}
    \phi^\lam(\cdot) = \phi(\cdot) + \frac{1}{2\lam}\|\cdot-x_0\|^2.
\end{equation}
We begin the analysis with the following technical result.

\begin{lemma}\label{lem:dist}
Let $F_1,F_2:\R^n \to \R$ be two $\mu$-strongly convex functions for some $\mu>0$ and their corresponding minimizers be $x_1^*$ and $x_2^*$. Assume that $F_1-F_2$ is an $L$-Lipschitz continuous function for some $L>0$, then $\|x_1^*-x_2^*\|\le L/\mu$.
\end{lemma}

\begin{proof}
Let $G:=F_1-F_2$. Using the $\mu$-strong convexity of $F_1$ and $F_2$, we have
\begin{align*}
    |G(x_1^*)-G(x_2^*)| &= |F_1(x_1^*)-F_2(x_1^*) - F_1(x_2^*) + F_2(x_2^*)| \\
    &= F_1(x_2^*) - F_1(x_1^*) + F_2(x_1^*) - F_2(x_2^*)\\
    &\ge \frac{\mu}{2}\|x_2^* - x_1^*\|^2 + \frac{\mu}{2}\|x_1^* - x_2^*\|^2
    = \mu\|x_2^* - x_1^*\|^2.
\end{align*}
It follows from the above inequality and the $L$-Lipschitz continuity of $G$ that
\[
\mu\|x_2^* - x_1^*\|^2\le |G(x_1^*)-G(x_2^*)| \le L \|x_2^* - x_1^*\|,
\]
and hence the lemma holds.
\end{proof}

Next, we present a bound on the distance between consecutive iterates $x_{j-1}$ and $x_j$. 

\begin{lemma}\label{lem:xj}
Letting $\Gamma_1(\cdot) = \ell_f(\cdot;x_0)$ and assuming that $\Gamma_{j+1}$ is $M$-Lipschitz continuous and satisfies \eqref{ineq:Gamma1} for every $j\ge 1$. Then, we have $\|x_j-x_{j-1}\| \le 2 \lam M$ for every $j\ge 2$. 
\end{lemma}

\begin{proof}
For any $j\ge 2$, we consider two functions 
\[
    F_1(u):=\Gamma_{j-1}(u) + h(u) +\frac{1}{2\lam}\|u- x_0 \|^2, \quad F_2(u):=\Gamma_j(u) + h(u) +\frac{1}{2\lam}\|u- x_0\|^2.
\]
It is clear that they are both $\lam^{-1}$-strongly convex.
Moreover, it follows from the assumption that both $\Gamma_{j-1}$ and $\Gamma_j$ are $M$-Lipschitz continuous that $G:=F_1-F_2$ is $2M$-Lipschitz continuous. Indeed, we first observe that
\[
G(x)=F_1(x)-F_2(x) = \Gamma_{j-1}(x)-\Gamma_j(x),
\]
and hence have
\begin{align*}
    |G(x)-G(y)|
    &=|\Gamma_{j-1}(x)-\Gamma_j(x) - [\Gamma_{j-1}(y)-\Gamma_j(y)]|\\
    &\le |\Gamma_{j-1}(x)-\Gamma_{j-1}(y)| + |\Gamma_j(x)-\Gamma_j(y)|
    \le 2M\|x-y\|.
\end{align*}
Hence, $F_1$ and $F_2$ satisfy the assumptions in Lemma \ref{lem:dist} with $\mu=\lam^{-1}$ and $L=2M$.
Since $x_{j-1}=\argmin_{u\in \R^n} F_1(u)$ and $x_j=\argmin_{u\in \R^n} F_2(u)$ by \eqref{eq:xj},
the conclusion immediately follows from Lemma \ref{lem:dist}.
\end{proof}

The following result provides a tighter bound than the one in Proposition \ref{prop:inner}, and therefore will lead to improved bounds in Proposition~\ref{prop:spp-inner}. 

\begin{proposition}\label{prop:tj-improved}
    For every $j\ge 2$, we define
    \begin{equation}\label{def:hx}
        \hat x_j = \begin{cases}    x_2, & \text{if} ~ j=2;\\
        \frac{3 x_2 + \sum_{i=3}^j i x_i}{A_j}, & \text{otherwise,}
        \end{cases}
    \end{equation}
    where $A_j=j(j+1)/2$, and
    \begin{equation}\label{def:ht}
        \hat t_j = \phi^\lam(\hat x_j) - m_j,
    \end{equation}
    where $m_j$ is as in \eqref{eq:mj} and $\phi^\lam$ is as in \eqref{eq:phi2}.
    Then, we have for every $j \ge 2$,
    \begin{equation}\label{ineq:ht}
        \hat t_j \le \frac{16 \lam M^2}{j+1}.
    \end{equation}
\end{proposition}

\begin{proof}
Using the definitions of $t_j$ and $m_j$ in \eqref{eq:mj} and the inequality in \eqref{ineq:txj} with $j=1$, we have
	\[
	t_2
	\stackrel{\eqref{eq:mj}}=\phi^\lam(\tx_2)-m_2
	\stackrel{\eqref{ineq:txj},\eqref{eq:mj}}\le \phi^\lam(x_2)- \left[(\Gamma_2+ h)(x_2) +\frac{1}{2 \lam}\|x_2-x_0\|^2\right]
	\stackrel{\eqref{eq:phi2}}=  f(x_2)-\Gamma_2(x_2),
	\]
    where the last identity is due to the definition of $\phi^\lam$ in \eqref{eq:phi2}.
	It follows from \eqref{ineq:Gamma1} and the definition of $\ell_f$ in \eqref{linfdef} that
	\[
	\Gamma_2(\cdot) \stackrel{\eqref{ineq:Gamma1}}\ge \ell_f(\cdot;x_1) \stackrel{\eqref{linfdef}}= f(x_1) + \inner{f'(x_1)}{\cdot-x_1}.
	\]
	Combining the above two inequalities, we obtain
	\begin{align*}
		t_2&\le f(x_2) - [ f(x_1) + \inner{f'(x_1)}{x_2 - x_1}] \\
		&\le |f(x_2)-f(x_1)|+\|f'(x_1)\| \|x_2-x_1\| \le 2M \|x_2-x_1\|,
	\end{align*}
	where the second inequality is due to the triangle and the Cauchy-Schwarz inequalities, and the last inequality follows from assumption (A2).
    Hence, it follows from Lemma \ref{lem:xj} that $t_2\le 4 \lam M^2$.
    Since \eqref{ineq:Gamma1} implies \eqref{def:Gammaj}, following an argument similar to the proof of Proposition \ref{prop:inner}, we have
        \begin{align}
        A_j m_j
        &\ge A_2 m_2 + 3 \phi^\lam(x_3) + \dots + j\phi^\lam(x_j) - 8\lam M^2 (j-2) \nn \\
        &\stackrel{\eqref{eq:mj}}{=} -A_2 t_2 + A_2 \phi^\lam(x_2) + 3 \phi^\lam(x_3) + \dots + j\phi^\lam(x_j) - 8\lam M^2 (j-2) \nn \\
        &\stackrel{\eqref{def:hx}}{\ge} -A_2 t_2 + A_j \phi^\lam(\hat x_j) - 8\lam M^2 (j-2), \label{ineq:A-new}
        \end{align}
    where the last inequality is due to the convexity of $\phi^\lam$ and the definition of $\hat x_j$ in \eqref{def:hx}.
    Using the definition of $\hat t_j$ in \eqref{def:ht} and the facts that $t_2\le 4\lam M^2$ and $A_j=j(j+1)/2$, we obtain
        \[
        A_j \hat t_j \stackrel{\eqref{def:ht}}= A_j(\phi^\lam(\hat x_j) - m_j) \stackrel{\eqref{ineq:A-new}}\le A_2 t_2 + 8\lam M^2 (j-2) \le 12 \lam M^2 + 8\lam M^2 j - 16\lam M^2 \le 8\lam M^2 j.
        \]
    Therefore, inequality \eqref{ineq:ht} immediately follows.
\end{proof}

Proposition~\ref{prop:tj-improved} is the key result needed to derive an improved version of Proposition~\ref{prop:spp-inner}. The remaining steps follow similarly by formally redefining the relevant quantities using $\hat x_j$ (defined in \eqref{def:hx}) in place of $\tx_j$. To avoid introducing additional notation and repeating arguments, we directly state the resulting bounds:
\begin{equation}\label{ineq:hat-tj}
    \hat t_k^x\le \frac{16\lam_k M^2}{l_k+1}, \quad \hat t_k^y\le \frac{16\lam_k M^2}{l_k+1},
\end{equation}
where $\lam_k$ and $l_k$ are as in Proposition~\ref{prop:spp-inner}, and $\hat t_k^x$ and $\hat t_k^y$ are the counterparts of $t_k^x$ and $t_k^y$ used in Proposition~\ref{prop:spp-inner}, but with $\tx_k$ and $\ty_k$ replaced by $\hat x_k$ and $\hat y_k$ in their definition \eqref{def:tj-xy}.

By making an assumption analogous to \eqref{ineq:spp-accuracy}, namely,
\begin{equation}\label{ineq:spp-accuracy1}
    \hat t_k^x \le \frac{\bar \varepsilon}4, \quad
    \hat t_k^y \le \frac{\bar \varepsilon}4,
\end{equation}
we are able to reproduce similar versions of Lemma \ref{lem:spp-outer}, Lemma \ref{lem:spp-outer-comb}, and Proposition~\ref{prop:pb-spp}.
We are now ready to establish the improved iteration-complexity ${\cal O}((MD/\bar \varepsilon)^2)$ for PB-SPP to find a $\bar \varepsilon$-saddle-point.

\begin{theorem}\label{thm:pb-spp1}
Given $(x_0,y_0,\bar \varepsilon) \in \dom h_1 \times \dom h_2 \times R_{++}$, letting $\lam_1 = D/(4M)$, then the iteration-complexity for PB-SPP$(x_0,y_0)$ to find a $\bar \varepsilon$-saddle-point of \eqref{eq:spp} is ${\cal O}((MD/\bar \varepsilon)^2)$.
\end{theorem}

\begin{proof}
Following an argument analogous to the proof of Theorem \ref{thm:pb-spp}, we can show that PB-SPP takes $k=64M^2D^2/\bar \varepsilon^2$ iterations to find the $\bar \varepsilon$-saddle-point.
Using \eqref{ineq:hat-tj}, we know to have \eqref{ineq:spp-accuracy1} holds for every cycle ${\cal C}_i$, it is sufficient to have
\[
l_i = \frac{64\lam_i M^2}{\bar \varepsilon} = \frac{16M D}{\bar \varepsilon \sqrt{i}},
\]
where the second identity is due to the facts that $\lam_i=\lam_1/\sqrt{i}$ and $\lam_1 = D/(4M)$.
As a consequence, the total number of iterations  is
\[
\sum_{i=1}^k l_i = \sum_{i=1}^k \frac{16M D}{\bar \varepsilon \sqrt{i}} \le \frac{256 M^2 D^2}{\bar \varepsilon^2},
\]
where we use the facts that$\sum_{i=1}^k (1/\sqrt{i}) \le \int_0^k (1/\sqrt{x}) dx = 2 \sqrt{k}$ and $k=64M^2D^2/\bar \varepsilon^2$.
\end{proof}

\section{Numerical experiments}\label{sec:numerics}

We consider the following regularized matrix game
\begin{equation}\label{def:linf_reg}
    \min_{x\in\Delta_n}\max_{y\in\Delta_m}\{y^\top Ax + \gamma_x\infnorm{x}-\gamma_y\infnorm{y}\},
\end{equation}
where $A\in\R^{m\times n}$ is the payoff matrix, $x$ and $y$ are mixed strategies on unit simplices $\Delta_n$ and $\Delta_m$, respectively. 
The $\ell_\infty$ regularization terms with parameters $\gamma_x \ge 0$ and $\gamma_y \ge 0$ discourage overly concentrated strategies by penalizing large coordinates, thereby promoting robustness.
Note that \eqref{def:linf_reg} is in the form of SPP \eqref{eq:spp} with
\begin{equation}\label{eq:fh}
    f(x,y)=y^\top Ax + \gamma_x\infnorm{x}-\gamma_y\infnorm{y}, \quad h_1(x) = I_{\Delta_n}(x), \quad h_2(y) = I_{\Delta_m}(y),
\end{equation}
where $I_{\Delta_n}$ and $I_{\Delta_m}$ are the indicator functions of unit simplicies $\Delta_n$ and $\Delta_m$, respectively.

The subgradient $f'_x$ and the supergradient $f'_y$ are given by
\begin{equation}
    f'_x(u,v)=A^\top v+\gamma_x g_u, \quad f'_y(u,v)=Au-\gamma_y g_v
\end{equation}
where $g_u\in\partial \|u\|_\infty$ and $g_v\in\partial \|v\|_\infty$. It follows from Example 3.52 of \cite{Beck2017} that the subdifferential of $\infnorm{\cdot}$ takes the form of
\begin{equation}\label{eq:infnorm}
    \partial \infnorm{x}=\left\{\sum_{j\in\mathcal{I}(x)}\lambda_je_j:\lambda\in\Delta^n, \sum_{j\not\in\mathcal{I}(x)}\lambda_i=0\right\},
\end{equation}
where $e_j$ is the $j$-th unit vector and the index set $\mathcal{I}(x)=\{j:|x_j|=\infnorm{x}\}$.
In our implementation, we fix $g_u=\sum_{j\in\mathcal{I}(u)}\lambda_je_j$ with $\lam_j=1/|I(u)|$ and $g_v=\sum_{j\in\mathcal{I}(v)}\lambda_je_j$ with $\lam_j=1/|I(v)|$.
We also note that
\[  M_x=\sup_{u\in\Delta_n,v\in\Delta_m}\|f'_x(u,v)\|=\sup_{u\in\Delta_n,v\in\Delta_m}\{\|A^\top v\| + \gamma_x\|g_u\|\} \leq \max_{1\leq j\leq m}\|A^\top_{j}\| +\gamma_x
\]
and 
\[
M_y=\sup_{u\in\Delta_n,v\in\Delta_m}\|f'_y(u,v)\|=\sup_{u\in\Delta_n,v\in\Delta_m}\{\|Au\| + \gamma_y\|g_v\|\} \leq \max_{1\leq i\leq n}\|A_{i}\| +\gamma_y,
\]
where $A_j^\top$ (resp., $A_i$) denotes the $j$-th (resp., $i$-th) column of $A^\top$ (resp., $A$). 
Indeed, the above inequalities follow from \eqref{eq:infnorm}, that is
\[
\|g_u\|^2 = \sum_{i\in I(x)} \lam_i^2 \le \sum_{i\in I(x)} \lam_i \le 1,
\]
and similarly $\|g_v\|\le 1$. Clearly, taking $M= \max\{M_x,M_y\}$ satisfies \eqref{ineq:spp-Lips}.

In the regularized matrix game \eqref{def:linf_reg},
we set $m=n=100$ and $\gamma_x=\gamma_y=0.05$, and generate the payoff matrix $A$ of $5\%$ density with nonzero entries sampled from $\mathcal{N}(0,1)$. 
We compare four numerical methods on \eqref{def:linf_reg}: CS-SPP (i.e., \eqref{eq:spp-xk}-\eqref{eq:spp-yk}), and three variants of PB-SPP, where the bundle model $\Gamma_k^x$ (resp., $-\Gamma_k^y$) in \eqref{eq:xk-spp} (resp., \eqref{eq:yk-spp}) is generated by the one-cut scheme \eqref{eq:one-cut}, the two-cuts scheme \eqref{eq:tc}, and the multiple-cuts scheme \eqref{eq:mc}, respectively.
For the two-cuts scheme, the Lagrange multiplier $\theta_{j-1}$ in \eqref{eq:tc-bar} is obtained via a bisection search for an auxiliary problem, while in the multiple-cuts scheme, $\theta_j^i$ in \eqref{eq:zjtheta} is computed by solving an auxiliary problem using FISTA. All methods are implemented in Julia. Proximal mappings for $h_1$ and $h_2$ in \eqref{eq:fh} are evaluated using the ProximalOperators.jl package, and the FISTA routine for the multiple-cuts scheme is taken from the ProximalAlgorithms.jl package.

We set $x_0 = (1/n, \dots, 1/n)^\top \in \R^n$ and $y_0 = (1/m, \dots, 1/m)^\top \in \R^m$ and use $(x_0,y_0)$ as the initial pair for each method.
We set tolerance $\bar \varepsilon=10^{-4}$, the static stepsize $\lam = \bar \varepsilon/(32M^2)$ for CS-SPP, and the dynamic stepsize $\lam_k = D/(4M \sqrt{k})$ with $D=2$ for $k \ge 1$, which is used by all three variants of PB-SPP.
All numerical methods in the benchmark are terminated once a $\bar \varepsilon$-saddle-point, as defined in \eqref{def:eps-spp-equiv}, is obtained. From the definitions of $\varphi$ and $\psi$ in \eqref{def:spp-func}, it follows that for each $x\in \Delta_n$ and $y\in \Delta_m$, 
\begin{equation}\label{eq:varphi}
    \varphi(x) = \gamma_x\infnorm{x} + \max _{y \in \Delta_m} \left\{y^\top Ax -\gamma_y\infnorm{y}\right\}, \,
    \psi(y) = -\gamma_y\infnorm{y} + \min _{x \in \Delta_n} \left\{y^\top Ax + \gamma_x\infnorm{x}\right\}.
\end{equation}
Evaluating $\varphi$ or $\psi$ requires an exact solution to a generic optimization problem of the form
\begin{equation}\label{def:fz}
    \min_{x \in \Delta_n} \left\{f_z(x)=z^\top x + \gamma\infnorm{x}\right\}.
\end{equation}
Algorithm~\ref{alg:Sj} in Appendix~\ref{sec:detail} provides a numerical scheme for the exact solution to this problem.

We track the primal-dual gap along with the elapsed time, the total number of proximal evaluations,  and the number of outer iterations. CS-SPP logs this information every 1000 iterations (since the iterations are both much more numerous and much faster), while PB-SPP logs every 10 iterations. Numerical tests are conducted on an i9-13900k desktop with 64 GB of RAM. 

\begin{figure}[H]
    \centering
    \includegraphics[width=\linewidth]{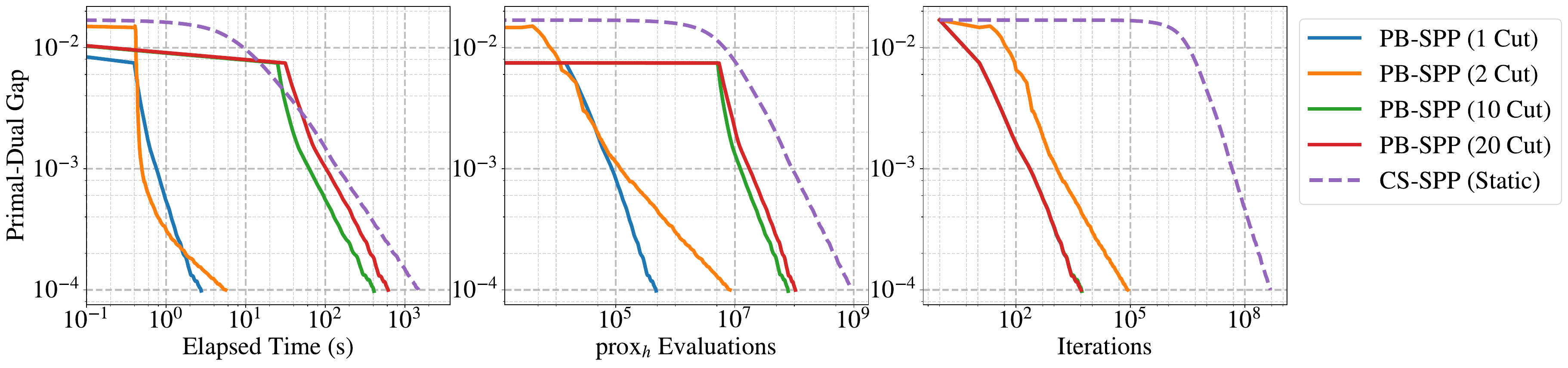}
    \caption{Comparison between CS-SPP and PB-SPP with one-cut, two-cuts, and multi-cuts schemes for solving~\eqref{def:linf_reg}.}
    \label{fig:SPP}
\end{figure}

Figure \ref{fig:SPP} compares five methods for solving \eqref{def:linf_reg}: CS-SPP with a static stepsize of $\bar{\varepsilon}/(32M^2)$, and PB-SPP with a dynamic stepsize of $1/(2M\sqrt{k})$ under one-cut, two-cuts, 10-cuts, and 20-cuts schemes.
Among these, the one-cut and two-cut PB-SPP schemes are the most efficient in terms of elapsed time. The multi-cut schemes with 10 or 20 cuts show nearly identical performance across all metrics: elapsed time, number of proximal evaluations, and iteration counts. Regarding the total number of (PDCP) iterations, the two-cuts scheme requires the fewest iterations, followed by the multi-cuts schemes, and finally the one-cut scheme.

\section{Concluding remarks} \label{sec:conclusion}

This paper studies the iteration-complexity of modern PB methods for solving CNCO \eqref{eq:problem} and SPP \eqref{eq:spp}. It proposes PDPB for solving \eqref{eq:problem} and provides the iteration-complexity of PDPB in terms of a primal-dual gap. The paper also introduces PB-SPP for solving \eqref{eq:spp} and establishes the iteration-complexity to find a $\bar \varepsilon$-saddle-point. 
Another interesting feature of the paper is that it investigates the duality between CG and PDCP for solving the proximal subproblem \eqref{eq:phi-lam}. The paper further develops novel variants of both CG and PDCP leveraging the duality.

We finally discuss some possible extensions of our methods and analyses.
First, we have studied modern PB methods for solving CNCO and SPP in this paper, and we could extend the methods to solving more general nonsmooth problems with convex structures such constrained optimization, equilibrium problems, and variational inequalities.
Second, it is interesting to study the duality between PDCP and CG in the context of SPP, which is equivalent to developing a CG method to implement \eqref{incl} and \eqref{ineq:IPP} within IPPF.
Third, similar to the universal methods proposed in \cite{guigues2024universal}, we are also interested in developing universal variants of PB-SPP for SPP \eqref{eq:spp} under strong convexity assumptions without knowing the problem-dependent parameters a priori.
Finally, following the stochastic PB method developed for stochastic CNCO in \cite{liang2024single}, it is worthwhile to explore stochastic versions of PB-SPP for solving stochastic SPP, particularly those involving decision-dependent distributions.

\bibliographystyle{plain}
\bibliography{ref}

\appendix
\section{Technical results and deferred proofs}\label{sec:technical}

This section collects technical results used throughout the paper and deferred proofs from Section~\ref{sec:PB-SPP}.

\subsection{Technical results}

We present Lemma 13.7 of \cite{Beck2017} with slight modification, which is used in the proof of Lemma~\ref{lem:FW-basic}.

\begin{lemma}\label{lem:Beck-FW}
Consider
\begin{equation}\label{eq:Beck}
    \min_{x \in \R^n} \{F(x) = f(x) + g(x)\},
\end{equation}
where $f \in \bConv{n}$, $g \in \bConv{n}$, and $\dom g \subset \dom f$.
Moreover, $f$ is $L_f$-smooth over $\dom f$.
Define
\[
S(x)=\max _{p \in \R^n}\{\inner{\nabla f(x)}{x-p}+g(x)-g(p)\}, \quad p(x) = \underset{p\in \R^n}\argmin \{\inner{p}{\nabla f(x)} + g(p)\}.
\]
Then, for every $x \in \dom g$ and $t \in [0, 1]$, if $p(x)$ exists, we have
\begin{equation}\label{ineq:Beck-basic}
    F(x+t(p(x) -x)) \le F(x)-t S(x)+\frac{t^2 L_f}{2}\|p(x)-x\|^2.
\end{equation}
\end{lemma}

\begin{lemma}\label{lem:spp-equiv}
    Given $\bar \varepsilon>0$, a pair $(x, y)$ is a $\bar \varepsilon$-saddle-point of \eqref{eq:spp} (i.e., satisfying \eqref{def:eps-spp}) if and only if the pair satisfies \eqref{def:eps-spp-equiv}.
\end{lemma}

\begin{proof}
    It follows from \eqref{def:eps-spp} that for every $(u,v)\in \dom h_1 \times \dom h_2$,
    \begin{equation}\label{ineq:phi-subdiff}
        \phi(u,y)-\phi(x,v) \ge \phi(x,y)-\phi(x,y) -\bar \varepsilon = -\bar \varepsilon.
    \end{equation}
    Hence, \eqref{ineq:phi-subdiff} holds with $(u,v)=(x(y),y(x))$ where 
    \[
    x(y) = \underset{x \in \R^n}\argmin \phi(x, y), \quad y(x) = \underset{y \in \R^m}\argmax \phi(x, y),
    \]
    that is
    \[
    \underset{x \in \R^n}\min \phi(x, y) - \underset{y \in \R^m}\max \phi(x, y) = \phi(x(y),y)-\phi(x,y(x)) \stackrel{\eqref{ineq:phi-subdiff}}\ge -\bar \varepsilon.
    \]
    This result, together with \eqref{eq:spp-Phi} and \eqref{def:spp-func}, implies that \eqref{def:eps-spp-equiv} holds.
    On the other hand, assuming that \eqref{def:eps-spp-equiv} holds, then for every $(u,v)\in \dom h_1 \times \dom h_2$, it obviously follows from \eqref{def:spp-func} that
    \[
    \phi(x,v)-\phi(u,y) \stackrel{\eqref{def:spp-func}}\le \varphi(x)-\psi(y) \le \bar \varepsilon,
    \]
    which is \eqref{def:eps-spp} in view of \eqref{ineq:phi-subdiff}.
\end{proof}

\begin{lemma}\label{lem:101}
    Given $\bar \varepsilon>0$, a pair $(x, y)$ is a $\bar \varepsilon$-saddle-point of \eqref{eq:spp} (i.e., satisfying \eqref{def:eps-spp}) implies \eqref{def:spp-weak}.
\end{lemma}

\begin{proof}
Assuming  that $(x, y)$ is a $\bar \varepsilon$-saddle-point, it follows from Lemma \ref{lem:spp-equiv} that \eqref{def:eps-spp-equiv} holds, and hence that for every $(u,v)\in \dom h_1 \times \dom h_2$,
\begin{equation}\label{ineq:phi-symm}
    \phi(x,v)-\phi(u,y) \stackrel{\eqref{def:spp-func}}\le \varphi(x)-\psi(y) \le \bar \varepsilon,
\end{equation}
where the first inequality is due to \eqref{def:spp-func}.
Taking $(u,v)=(x_*,y)$ in \eqref{ineq:phi-symm} and using the first inequality in \eqref{def:spp}, we have
\[
\phi(x,y) - \phi(x_*,y_*) \stackrel{\eqref{def:spp}}\le \phi(x,y)-\phi(x_*,y) \stackrel{\eqref{ineq:phi-symm}}\le \bar \varepsilon.
\]
Taking $(u,v)=(x,y_*)$ in \eqref{ineq:phi-symm} and using the second inequality in \eqref{def:spp}, we have
\[
\phi(x_*,y_*) - \phi(x,y) \stackrel{\eqref{def:spp}}\le \phi(x,y_*) - \phi(x,y) \stackrel{\eqref{ineq:phi-symm}}\le \bar \varepsilon.
\]
Therefore, \eqref{def:spp-weak} immediately follows from the above two inequalities.
\end{proof}

\subsection{Proof of Proposition \ref{prop:CS-SPP-IPPF}} \label{subsec:pf-CS-IPPF}

\begin{proof}
We first show that CS-SPP satisfies \eqref{incl}.
It follows from the CS-SPP iterate \eqref{eq:spp-xk} that
\[
\frac{x_{k-1}-x_k}{\lam} \in \partial [\ell_{f(\cdot,y_{k-1})}(\cdot;x_{k-1})+h_1](x_k).
\]
Using the inclusion above, we have for every $u\in \dom h_1$,
\[
[\ell_{f(\cdot,y_{k-1})}(\cdot;x_{k-1})+h_1](u) \ge [\ell_{f(\cdot,y_{k-1})}(\cdot;x_{k-1})+h_1](x_k) + \frac{1}{\lam}\inner{x_{k-1}-x_k}{u-x_k}.
\]
Using the definition of $p_k$ in \eqref{def:p-d-k} and the fact that $f(\cdot,y_{k-1})$ is convex, we further obtain
\[
    p_k(u) \ge p_k(x_k) + \frac{1}{\lam}\inner{x_{k-1}-x_k}{u-x_k} - \varepsilon_k^x,
\]
where $\varepsilon_k^x$ is as in \eqref{def:ekx}.
Similarly, we have for every $v\in \dom h_2$,
\[
    d_k(v) \ge d_k(y_k) + \frac{1}{\lam}\inner{y_{k-1}-y_k}{v-y_k} - \varepsilon_k^y,
\]
where $\varepsilon_k^y$ is as in \eqref{def:eky}. 
Summing the above two inequalities gives \eqref{ineq:target} with $\lam_k=\lam$, $\varepsilon_k = \varepsilon_k^x + \varepsilon_k^y$ and $(\tx_k,\ty_k) = (x_k,y_k)$, and hence \eqref{incl} holds in view of Lemma \ref{lem:pd-equiv}.

We next show that CS-SPP satisfies \eqref{ineq:IPP}.
Indeed, it follows from the definition of $\varepsilon_k^x$ in \eqref{def:ekx} and the first inequality in \eqref{ineq:Lips-spp} that
\begin{align*}
    2\lam \varepsilon_k^x - \|x_k-x_{k-1}\|^2
    &\stackrel{\eqref{def:ekx}}= 2\lam [f(x_k,y_{k-1}) - \ell_{f(\cdot,y_{k-1})}(x_k;x_{k-1})] - \|x_k-x_{k-1}\|^2\\
    &\stackrel{\eqref{ineq:Lips-spp}}\le 4 \lam M \|x_k-x_{k-1}\| - \|x_k-x_{k-1}\|^2 \le 4\lam^2 M^2.
\end{align*}
Similarly, we have $2\lam \varepsilon_k^y - \|y_k-y_{k-1}\|^2 \le 4\lam^2 M^2$.
Summing the two inequalities and using the facts that $\lam=\sqrt{\delta/8M^2}$ and $\varepsilon_k = \varepsilon_k^x + \varepsilon_k^y$, we have
\[
2\lam \varepsilon_k - \|x_k-x_{k-1}\|^2 - \|y_k-y_{k-1}\|^2 \le 8\lam^2 M^2 = \delta,
\]
which is \eqref{ineq:IPP} with $\sigma=1$, $(\lam_k,\delta_k)=(\lam,\delta)$, and $(\tx_k,\ty_k) = (x_k,y_k)$.
\end{proof}

\subsection{Proof of Proposition \ref{prop:PB-SPP-IPPF}} \label{subsec:pf-PB-IPPF}

\begin{proof}
We first show that PB-SPP satisfies \eqref{incl}.
It follows from \eqref{eq:xk-spp} that
\[
\frac{x_{k-1}-x_k}{\lam_k} \in \partial (\Gamma_k^x+h_1)(x_k),
\]
which implies that for every $u \in \dom h_1$,
\[
    (\Gamma_k^x+h_1)(u) \ge (\Gamma_k^x+h_1)(x_k) + \frac{1}{\lam_k}\inner{x_{k-1}-x_k}{u-x_k}.
\]
Using the first inequality in \eqref{ineq:cvx} and the definition of $p_k$ in \eqref{def:p-d-k}, we have
\[
    p_k(u) \ge (\Gamma_k^x+h_1)(u) \ge p_k(\tx_k) + \frac{1}{\lam_k}\inner{x_{k-1}-x_k}{u-\tx_k} - \varepsilon_k^x, \quad \forall u,
\]
where
$\varepsilon_k^x$ is as in \eqref{def:ekx-pb}.
Similarly, we have for every $v \in \dom h_2$,
\[
    d_k(v) \ge d_k(\ty_k) + \frac{1}{\lam_k}\inner{y_{k-1}-y_k}{v-\ty_k} - \varepsilon_k^y, \quad \forall v,
\]
where
$\varepsilon_k^y$ is as in \eqref{def:eky-pb}.
Summing the above two inequalities gives \eqref{ineq:target} with $\varepsilon_k = \varepsilon_k^x + \varepsilon_k^y$, and hence \eqref{incl} holds in view of Lemma \ref{lem:pd-equiv}.

We next show that PB-SPP satisfies \eqref{ineq:IPP}.
Indeed, it follows from the definitions of $\varepsilon_k^x$ and $\varepsilon_k^y$ in \eqref{def:ekx-pb} and \eqref{def:eky-pb}, respectively, that
\[
\|x_k-\tx_k\|^2 +\|y_k-\ty_k\|^2+2 \lam_k \varepsilon_k  = \lam_k \left(  p_k^\lam(\tilde x_k) - m_k^x + d^\lam_k(\ty_k) -m_k^y\right),
\]
where $p_k^\lam$ and $d_k^\lam$ are as in \eqref{def:p-d} and $m_k^x$ and $m_k^y$ as the optimal values of \eqref{eq:xk-spp} and \eqref{eq:yk-spp}, respectively.
In view of \eqref{def:tj-xy} and \eqref{ineq:spp-accuracy}, the above relation further implies that
\[
\|x_k-\tx_k\|^2 +\|y_k-\ty_k\|^2+2 \lam_k \varepsilon_k \le \frac{\lam_k \bar \varepsilon}{2},
\]
which is \eqref{ineq:IPP} with $\sigma=0$ and $\delta_k=\lam_k \bar \varepsilon/2$.
\end{proof}

\section{Primal-dual subgradient method for CNCO} \label{sec:PDS}

This section is devoted to the complexity analysis of PDS. The main result is Theorem~\ref{thm:pds} below.

Recall the definitions of $d_0$ and $x_0^*$ in \eqref{eq:d0}. Since $x_0^* \in B(\hat x_0, 4d_0)$, which is the ball centered at $\hat x_0$ and with radius $4d_0$, it is easy to see that to solve \eqref{eq:problem}, it suffices to solve
\begin{equation}\label{eq:constrained}
    \min \left\{\hat \phi(x):=f(x)+\hat h(x): x \in \R^n\right\}
    = \min \left\{\phi(x): x \in Q\right\},
\end{equation}
where $\hat h = h + I_Q$ and $I_Q$ is the indicator function of $Q = B(\hat x_0, 4d_0)$.
Hence, it is convenient to consider a slightly modified version of PDS$(\hat x_0,\lam)$ with $h$ replaced by $\hat h$ in \eqref{eq:pds}, denoted by MPDS$(\hat x_0,\lam)$, i.e.,
\begin{equation}\label{eq:pds-1}
    s_k = f'(\hat x_{k-1}), \quad \hat x_k = \underset{u\in\R^n}{\argmin}\left\{\ell_f(u;\hat x_{k-1}) + \hat h(u) + \frac{1}{2\lam}\|u-\hat x_{k-1}\|^2 \right\}.
\end{equation}
It is worth noting that MPDS$(\hat x_0,\lam)$ is a conceptual method since we do not know $d_0$ and hence $\hat h$. We show equivalence between PDS$(\hat x_0,\lam)$ and MPDS$(x_0,\lam)$, and only use MPDS$(\hat x_0,\lam)$ for analyzing the convegence.

We first establish the complexity of the primal-dual convergence of MPDS$(\hat x_0,\lam)$ for solving \eqref{eq:constrained}, and then we argue that MPDS$(\hat x_0,\lam)$ and PDS$(\hat x_0,\lam)$ generate the same primal and dual sequences $\{\hat x_k\}$ and $\{s_k\}$ before convergence (see Lemma \ref{pds-lemma2}). Therefore, we also give the complexity of PDS$(\hat x_0,\lam)$ for solving \eqref{eq:constrained}.

The following lemma is the starting point of the primal-dual convergence analysis.

\begin{lemma}\label{pds-lemma1}
Given $\hat x_0 \in \R^n$, for every $k \ge 1$ and $u \in \dom \hat h$, the sequence $\{\hat x_k\}$ generated by MPDS$(\hat x_0,\lam)$ satisfies
\begin{equation}\label{ineq:hat-phi-recur}
    \hat \phi(\hat x_k) - \ell_f(u;\hat x_{k-1}) - \hat h(u) \le 2\lam M^2 + \frac{1}{2\lam}\|u- \hat x_{k-1}\|^2 - \frac1{2\lam}\|u - \hat x_k\|^2.
\end{equation}
\end{lemma}
\begin{proof}
Noticing that the objective function in \eqref{eq:pds-1} is $\lam^{-1}$-strongly convex, it then follows from Theorem 5.25(b) of \cite{Beck2017} that for every $u \in \dom \hat h$,
\begin{equation}\label{ineq:ell-h-mk}
    \ell_f(u;\hat x_{k-1}) + \hat h(u) + \frac1{2\lam}\|u-\hat x_{k-1}\|^2 \ge m_k + \frac{1}{2\lam} \|u-\hat x_k\|^2,
\end{equation}
where $m_k = \ell_f(\hat x_k;\hat x_{k-1}) + \hat h(\hat x_k) + \|\hat x_k-\hat x_{k-1}\|^2/(2\lam)$.
Using \eqref{ineq:f} with $(x,y)=(\hat x_k,\hat x_{k-1})$, we have
\[
    \hat \phi(\hat x_k) - m_k = f(\hat x_k) - \ell_f(\hat x_k;\hat x_{k-1}) \stackrel{\eqref{ineq:f}}{\le} 2M\|\hat x_k - \hat x_{k-1}\| - \frac1{2\lam}\|\hat x_k-\hat x_{k-1}\|^2 \le 2\lam M^2,
\]
where the last inequality is due to Young's inequality $a^2+b^2\ge 2ab$.
Hence, \eqref{ineq:hat-phi-recur} follows from combining the above inequality and \eqref{ineq:ell-h-mk}.
\end{proof}

The next result presents the primal-dual convergence rate of MPDS$(\hat x_0,\lam)$.

\begin{lemma}\label{pds-lemma3}
For every $k\ge 1 $, define
\begin{equation}\label{pds-bar-xg}
    \bar x_k = \frac1k \sum_{i=1}^{k} \hat x_i, \quad \bar s_k = \frac1k \sum_{i=1}^k s_i.
\end{equation}
Then, we have for every $k\ge 1 $, the primal-dual gap of \eqref{eq:constrained} is bounded as follows,
\begin{equation}\label{pds-lemma3-0}
    \hat \phi(\bar x_k) + f^*(\bar s_k) + \hat h^*(-\bar s_k) \le 2\lam M^2 + \frac{8 d_0^2}{\lam k}.
\end{equation}
\end{lemma}

\begin{proof}
We first note that $\ell_f(\cdot;\hat x_{k-1}) \le f$ and hence $(\ell_f(\cdot;\hat x_{k-1}))^* \ge f^*$. 
Using this inequality and the fact that $\nabla \ell_f(u;\hat x_{k-1}) = s_k$ for every $u \in \R^n$, we have
\[
\ell_f(u;\hat x_{k-1}) =-[\ell_f(\cdot;\hat x_{k-1})]^*(s_k) + \inner{s_k}{u} \le-f^*(s_k) + \inner{s_k}{u}.
\]
It thus follows from Lemma \ref{pds-lemma1} that for every $u \in \dom \hat h$,
\[
\hat \phi(\hat x_k) + f^*(s_k) - \inner{s_k}{u} - \hat h(u) \stackrel{\eqref{ineq:hat-phi-recur}}{\le} 2\lam M^2  + \frac{1}{2\lam}\|u - \hat x_{k-1} \|^2 - \frac1{2\lam}\|u - \hat x_k\|^2.
\]
Replacing the index $k$ in the above inequality by $i$, summing the resulting inequality from $i=1$ to $k$, and using convexity of $\hat \phi$ and $f^*$, we obtain for every $u \in \dom \hat h$,
\[
\hat \phi(\bar x_k) + f^*(\bar s_k) + \inner{-\bar s_k}{u} - \hat h(u) \le 2\lam M^2 + \frac{1}{2\lam k}\|u - \hat x_0\|^2,
\]
where $\bar x_k$ and $\bar s_k$ are as in \eqref{pds-bar-xg}.
Maximizing over $u\in \dom \hat h$ on both sides of the above inequality, we have
\[
\hat \phi(\bar x_k) + f^*(\bar s_k) + \hat h^*(-\bar s_k) \le 2\lam M^2 + \frac{\max\{\|u-\hat x_0\|^2: u \in \dom \hat h\} }{2\lam k}.
\]
Therefore, \eqref{pds-lemma3-0} follows by using the fact that $\dom \hat h \subset Q= B(\hat x_0, 4d_0)$.
\end{proof}

The following theorem provides the complexity of MPDS$(\hat x_0,\lam)$ for solving \eqref{eq:constrained}.

\begin{theorem}\label{pds-theorem1}
Given $(\hat x_0,\bar \varepsilon) \in \R^n \times \R_{++}$, letting $\lam = \bar \varepsilon/(16M^2)$, then the number of iterations for MPDS$(\hat x_0,\lam)$ to generate a primal-dual pair $(\bar x_k, \bar s_k)$ as in \eqref{pds-bar-xg} such that $\hat \phi(\bar x_k) + f^*(\bar s_k) + \hat h^*(-\bar s_k) \le \bar \varepsilon$
is at most $256 M^2 d_0^2/\bar \varepsilon^2$.
\end{theorem}

\begin{proof}
It follows from Lemma \ref{pds-lemma3} with $\lam = \bar \varepsilon/(16M^2)$ and $k=16 d_0^2/(\lam \bar \varepsilon)$ that
\[
\hat \phi(\bar x_k) + f^*(\bar s_k) + \hat h^*(-\bar s_k) \le \frac{\bar \varepsilon}{8} + \frac{\bar \varepsilon}{2} < \bar \varepsilon.
\]
Therefore, the theorem immediately follows from plugging $\lam= \bar \varepsilon/(16M^2)$ into $k=16 d_0^2/(\lam \bar \varepsilon)$.
\end{proof}

The next lemma gives the boundedness of $\{\hat x_k\}$ generated by PDS$(\hat x_0,\lam)$ and shows that $\{\hat x_k\} \subset Q = B(\hat x_0,4d_0)$. This result is important since it reveals the equivalence between PDS and MPDS, which is useful in Theorem \ref{thm:pds} below.

\begin{lemma}\label{pds-lemma2}
    For every $k \le 256M^2 d_0^2/\bar \varepsilon^2$, the sequence $\{\hat x_k\}$ generated by PDS$(\hat x_0,\lam)$ with $\lam=\bar \varepsilon/(16 M^2)$ satisfies $\hat x_k \in Q$.
\end{lemma}

\begin{proof}
    Following an argument similar to the proof of Lemma \ref{pds-lemma1}, we can prove for every $u \in \dom h$,
    \[
    \phi(\hat x_k) - \ell_f(u;\hat x_{k-1}) - h(u) \le  2\lam M^2 - \frac1{2\lam}\|u - \hat x_k\|^2+ \frac{1}{2\lam}\|u-\hat x_{k-1}\|^2,
    \]
    which together with the fact that $\ell_f(\cdot;\hat x_{k-1}) \le f$ implies that
    \begin{align*}
        \phi(\hat x_k) - \phi(u) \le 2\lam M^2 - \frac1{2\lam}\|u - \hat x_k\|^2+ \frac{1}{2\lam}\|u-\hat x_{k-1}\|^2.
    \end{align*}
    Taking $u=x_0^*$ and using the fact that $\phi(\hat x_k)\ge \phi_*= \phi(x_0^*)$, we obtain
    \[
    \|\hat x_k - x_0^*\|^2 \le 4\lam^2 M^2 + \|\hat x_{k-1} - x_0^*\|^2.
    \]
    Summing the above inequality, we show that for every $k \ge 1$, $\{\hat x_k\}$ generated by PDS$(\hat x_0,\lam)$ satisfies
    \begin{equation}\label{ineq:dk}
        \|\hat x_k - x_0^*\|^2 \le d_0^2 + 4\lam^2 M^2 k.
    \end{equation}
    Using the triangle inequality and the fact that $\sqrt{a+b} \le \sqrt{a} + \sqrt{b}$ for $a,b \ge 0$, we have
    \[
    \|\hat x_k - \hat x_0\| \le \|\hat x_k - x_0^*\| + \|\hat x_0 - x_0^*\| \stackrel{\eqref{ineq:dk}}{\le} 2d_0 + 2\lam M \sqrt{k}. 
    \]
    It thus follows from the assumptions on $k$ and $\lam$ that
    \[
    \|\hat x_k - \hat x_0\| \le 2d_0 + \frac{\bar \varepsilon}{8M} \frac{16M d_0}{\bar \varepsilon} = 4d_0,
    \]
    and hence that $x_k \in Q = B(\hat x_0,4d_0)$.
\end{proof}

Finally, using the complexity of MPDS$(\hat x_0,\lam)$ for solving \eqref{eq:constrained} (i.e., Theorem \ref{pds-theorem1}), we are ready to establish that of PDS$(\hat x_0,\lam)$.

\begin{theorem}\label{thm:pds}
Given $(\hat x_0,\bar \varepsilon) \in \R^n \times \R_{++}$, letting $\lam = \bar \varepsilon/(16M^2)$, then the number of iterations for PDS$(\hat x_0,\lam)$ to generate $(\bar x_k, \bar s_k)$ such that $\hat \phi(\bar x_k) + f^*(\bar s_k) + \hat h^*(-\bar s_k) \le \bar \varepsilon$
is at most $256 M^2 d_0^2/\bar \varepsilon^2$.
\end{theorem}

\begin{proof}
In view of Lemma \ref{pds-lemma2}, for $\lam=\bar \varepsilon/(16 M^2)$ and $k \le 256M^2 d_0^2/\bar \varepsilon^2$, the sequence $\{\hat x_k\}$ generated by PDS$(\hat x_0,\lam)$ is the same as the one generated by MPDS$(\hat x_0,\lam)$. Hence, sequences $\{s_k\}$ generated by the two methods are also the same, that is, \eqref{eq:pds-1} is identical to \eqref{eq:pds}.
Therefore, we conclude that the same primal-dual convergence guarantee holds for PDS$(\hat x_0,\lam)$ as the one for MPDS$(\hat x_0,\lam)$ in Theorem \ref{pds-theorem1}.
\end{proof}

\section{Composite subgradient method for SPP} \label{sec:CS-SPP}

This section is devoted to the complexity analysis of CS-SPP. The main result is Theorem~\ref{thm:cs-spp} below.

\begin{lemma}\label{lem:tech-h1}
For every $k \ge 1$ and $(u, v) \in \R^n \times \R^m$, we have
    \begin{align}
        &p_k(x_k) - \ell_{f(\cdot,y_{k-1})}(u;x_{k-1}) - h_1(u) \le \delta_k^x  + \frac{1}{2\lam}\|x_{k-1} - u\|^2 - \frac1{2\lam}\|x_k - u\|^2, \label{ineq:p-1}\\
        &d_k(y_k) + \ell_{f(x_{k-1},\cdot)}(v;y_{k-1}) - h_2(v) \le \delta_k^y  + \frac{1}{2\lam}\|y_{k-1} - v\|^2 - \frac1{2\lam}\|y_k - v\|^2, \label{ineq:d-1}
    \end{align}
    where
    \begin{equation}\label{def:delta}
        \delta_k^x = 2M\|x_k - x_{k-1}\| - \frac{1}{2\lam}\|x_k-x_{k-1}\|^2, \quad 
        \delta_k^y = 2M\|y_k - y_{k-1}\| - \frac{1}{2\lam}\|y_k-y_{k-1}\|^2. 
    \end{equation}
\end{lemma}

\begin{proof}
We only prove \eqref{ineq:p-1} to avoid duplication. Inequality \eqref{ineq:d-1} follows similarly.
Since the objective in \eqref{eq:spp-xk} is $\lam^{-1}$-strongly convex, we have for every $u \in \R^n$,
\begin{equation}\label{ineq:spp-x}
    \ell_{f(\cdot,y_{k-1})}(u;x_{k-1}) + h_1(u) + \frac1{2\lam}\|u-x_{k-1}\|^2 \ge m_k^x + \frac{1}{2\lam} \|u-x_{k}\|^2,
\end{equation}
where $m_k^x$ denotes the optimal value of \eqref{eq:spp-xk}.
Using the definition of $p_k$ in \eqref{def:p-d-k}, we have
\[
    p_k(x_k) - m_k^x = f(x_k,y_{k-1}) -\ell_{f(\cdot,y_{k-1})}(x_k;x_{k-1}) - \frac{1}{2\lam}\|x_k-x_{k-1}\|^2.
\]
It thus follows from the first inequality in \eqref{ineq:Lips-spp} with $(u,x,y)=(x_k,x_{k-1},y_{k-1})$ the definition of $\delta_k^x$ in \eqref{def:delta} that
\[
p_k(x_k) - m_k^x \le \delta_k^x,
\]
which together with \eqref{ineq:spp-x} implies that \eqref{ineq:p-1}.
\end{proof}

For $k \ge 1$, denote
\begin{equation}\label{eq:s3}
    s_k = (s_k^x, s_k^y), \quad s_k^x = f_x'(x_{k-1},y_{k-1}), \quad s_k^y =- f_y'(x_{k-1},y_{k-1}).
\end{equation}

We also denote $w = (u,v)$ and $z_k = (x_k,y_k)$ for all $k \ge 0$.

\begin{lemma}
For every $(u,v) \in \R^n \times \R^m$ and $k \ge 1$, we have
    \begin{align}
    &p_k(x_k) + f(\cdot,y_{k-1})^*(s_k^x) - h_1(u) + d_k(y_k) + [-f(x_{k-1},\cdot)]^*(s_k^y) - h_2(v) - \inner{s_k}{w} \nn \\
    \le & \delta_k^x + \delta_k^y  + \frac{1}{2\lam}\|z_{k-1} - w\|^2 - \frac1{2\lam}\|z_{k} - w\|^2. \label{ineq:p-d}
    \end{align}
\end{lemma}

\begin{proof}
It follows from the second identity in \eqref{eq:s3} that for every $u \in \R^n$,
\[
\nabla \ell_{f(\cdot,y_{k-1})}(u;x_{k-1}) = s_k^x,
\]
which together with Theorem 4.20 of \cite{Beck2017} implies that
\[
\ell_{f(\cdot,y_{k-1})}(u;x_{k-1}) + [\ell_{f(\cdot,y_{k-1})}(\cdot;x_{k-1})]^*(s_k^x) = \inner{u}{s_k^x}.
\]
Clearly, $\ell_{f(\cdot,y_{k-1})}(\cdot;x_{k-1}) \le f(\cdot, y_{k-1})$ and hence $[\ell_{f(\cdot,y_{k-1})}(\cdot;x_{k-1})]^* \ge f(\cdot, y_{k-1})^*$.
This inequality and the above identity imply that
\[
    \ell_{f(\cdot,y_{k-1})}(u;x_{k-1})
    \le-f(\cdot,y_{k-1})^*(s_k^x) + \inner{s_k^x}{u}. 
\]
It thus follows from \eqref{ineq:p-1} that
\[
p_k(x_k) + f(\cdot,y_{k-1})^*(s_k^x) - \inner{s_k^x}{u} - h_1(u) \le \delta_k^x  + \frac{1}{2\lam}\|x_{k-1} - u\|^2 - \frac1{2\lam}\|x_{k} - u\|^2.
\]
Similarly, we have for every $v\in \R^m$,
\[
d_k(y_k) + [-f(x_{k-1},\cdot)]^*(s_k^y) - \inner{s_k^y}{v} - h_2(v) \le \delta_k^y  + \frac{1}{2\lam}\|y_{k-1} - v\|^2 - \frac1{2\lam}\|y_{k} - v\|^2.
\]
Finally, summing the above two inequalities and using \eqref{eq:s3} and the facts that $w = (u,v)$ and $z_k = (x_k,y_k)$, we conclude that \eqref{ineq:p-d} holds.
\end{proof}

\begin{lemma}\label{lem:tech-h}
For every $(u,v) \in \R^n \times \R^m$ and $k \ge 1$, we have
    \begin{align} \label{ineq:xy}
& h_1(x_k) + f(\cdot,y_k)^*(s_k^x) - h_1(u) + h_2(y_k) + [-f(x_{k-1},\cdot)]^*(s_k^y) - h_2(v) - \inner{s_k}{w}  \nn \\
    \le& 16\lam M^2 + \frac{1}{2\lam}\|z_{k-1} - w\|^2 - \frac1{2\lam}\|z_k - w\|^2.
\end{align}
\end{lemma}

\begin{proof}
Using \eqref{def:p-d-k} and \eqref{ineq:p-d}, we have for every $(u,v) \in \R^n \times \R^m$,
\begin{align}
    &h_1(x_k) + f(\cdot,y_{k-1})^*(s_k^x) - h_1(u) + h_2(y_k) + [-f(x_{k-1},\cdot)]^*(s_k^y) - h_2(v)  - \inner{g_k}{w}\nn \\
    \le & \delta_k^x + \delta_k^y  + \frac{1}{2\lam}\|z_{k-1} - w\|^2 - \frac1{2\lam}\|z_{k} - w\|^2 + f(x_{k-1},y_k)- f(x_k,y_{k-1}). \label{ineq:full}
\end{align}
It immediately follows from \eqref{ineq:spp-Lips} that
\begin{align*}
    f(x_{k-1},y_k)- f(x_k,y_{k-1}) &= f(x_{k-1},y_k) - f(x_k,y_k) + f(x_k,y_k)- f(x_k,y_{k-1})\\
    &\le M\|x_k - x_{k-1}\| + M\|y_k - y_{k-1}\|.
\end{align*}
Following from the definition of conjugate functions and \eqref{ineq:spp-Lips} again, we have
\begin{align*}
    f(\cdot,y_{k-1})^*(s_k^x) & = \max_x \{\inner{x}{s_k^x} - f(x,y_k) + f(x,y_k) - f(x,y_{k-1})\} \\
    & \ge \max_x \{\inner{x}{s_k^x} - f(x,y_k)\} -M\|y_k - y_{k-1}\| \\
    &= f(\cdot,y_k)^*(s_k^x)-M\|y_k - y_{k-1}\|.
\end{align*}
Similarly, we also have 
\[
f(x_{k-1},\cdot)^*(-s_k^y) \le f(x_k,\cdot)^*(-s_k^y) + M\|x_k - x_{k-1}\|.
\]
Plugging the above three inequalities into \eqref{ineq:full}, we obtain for every $(u,v) \in \R^n \times \R^m$,
\begin{align*}
&h_1(x_k) + f(\cdot,y_k)^*(s_k^x) - h_1(u) + h_2(y_k) + [-f(x_{k-1},\cdot)]^*(s_k^y) - h_2(v) - \inner{s_k}{w} \\
    \le& \delta_k^x + \delta_k^y  + \frac{1}{2\lam}\|z_{k-1} - w\|^2 - \frac1{2\lam}\|z_{k} - w\|^2 + 2M\|x_k-x_{k-1}\| + 2M\|y_k-y_{k-1}\|.
\end{align*}
Noting from the definitions in \eqref{def:delta} that
\begin{align*}
    & \delta_k^x + \delta_k^y + 2M\|x_k-x_{k-1}\| + 2M\|y_k-y_{k-1}\|\\
\stackrel{\eqref{def:delta}}{=} & 4M\|x_k - x_{k-1}\| - \frac{1}{2\lam}\|x_k-x_{k-1}\|^2 + 4M\|y_k - y_{k-1}\| - \frac{1}{2\lam}\|y_k-y_{k-1}\|^2 \\
\le ~& 16\lam M^2,
\end{align*}
we finally conclude that \eqref{ineq:xy} holds.
\end{proof}

The following lemma collects technical results revealing relationships about the averages defined in \eqref{def:spp-pt} below.

\begin{lemma}\label{lem:avg}
Define
\begin{equation}\label{def:spp-pt}
    \bar x_k = \frac1k \sum_{i=1}^k x_i, \quad \bar y_k = \frac1k \sum_{i=1}^k y_i, \quad \bar s_k^x = \frac1k \sum_{i=1}^k s_i^x, \quad \bar s_k^y = \frac1k \sum_{i=1}^k s_i^y.
\end{equation}
Then, the following statements hold for every $k \ge 1$:
\begin{itemize}
    \item[(a)]
    \[
        \frac1k \sum_{i=1}^k f(\cdot,y_i)^*(s_i^x) \ge f(\cdot,\bar y_k)^*(\bar s_k^x), \quad \frac1k \sum_{i=1}^k [-f(x_i,\cdot)]^*(s_i^y) \ge [-f(\bar x_k,\cdot)]^*(\bar s_k^y);
    \]
    \item[(b)]
    \begin{align*}
        \varphi(\bar x_k) &\le h_1(\bar x_k) +[-f(\bar x_k, \cdot)]^*(\bar s_k^y) + h_2^*(-\bar s_k^y), \\
        -\psi(\bar y_k)&\le h_2(\bar y_k) + f(\cdot,\bar y_k)^*(\bar s_k^x) + h_1^*(-\bar s_k^x).
    \end{align*}
\end{itemize}
    
\end{lemma}

\begin{proof}
a) We only prove the first inequality to avoid duplication. The second one follows similarly. It follows from the definition of conjugate functions, \eqref{def:spp-pt}, concavity of $f(x,\cdot)$, and basic inequalities that
\begin{align*}
    \frac1k \sum_{i=1}^k f(\cdot,y_i)^*(s_i^x) & = \frac1k \sum_{i=1}^k \max_{x\in \R^n} \{\inner{x}{s_i^x} - f(x,y_i)\} \\
    & \ge \max_{x\in \R^n} \left\{\frac1k \sum_{i=1}^k \inner{x}{s_i^x} - \frac1k \sum_{i=1}^k f(x,y_i)\right\} \\
    & \stackrel{\eqref{def:spp-pt}}{\ge} \max_{x\in \R^n} \left\{\inner{x}{\bar s_k^x} - f(x,\bar y_k)\right\} 
    = f(\cdot,\bar y_k)^*(\bar s_k^x).
\end{align*}

b) For simplicity, we only prove the first inequality. The second one follows similarly.
It follows from the definition of $\varphi$ in \eqref{def:spp-func}, basic inequalities, and the definition of conjugate functions that
\begin{align*}
    \varphi(\bar x_k)& \stackrel{\eqref{def:spp-func}}{=}\max _{y \in \R^m} \phi(\bar x_k, y) =  h_1(\bar x_k)+ \max _{y \in \R^m} \{f(\bar x_k, y)-h_2(y)\} \\
    &\le h_1(\bar x_k) + \max _{y \in \R^m} \{\inner{y}{\bar s_k^y} - (-f(\bar x_k, y))\} + \max _{y \in \R^m} \{\inner{y}{-\bar s_k^y} -h_2(y)\} \\
    &= h_1(\bar x_k) +[-f(\bar x_k, \cdot)]^*(\bar s_k^y) + h_2^*(-\bar s_k^y).
\end{align*}
\end{proof}

\begin{proposition}\label{prop:spp-rate}
For every $k \ge 1$, we have
    \begin{equation}\label{ineq:Phi}
        \Phi(\bar x_k,\bar y_k) = \varphi(\bar x_k)-\psi(\bar y_k) \le 16 \lam M^2 + \frac{D^2}{2\lam k}
    \end{equation}
    where $\Phi(\cdot,\cdot)$ in as in \eqref{eq:spp-Phi}.
\end{proposition}

\begin{proof}
Replacing the index $k$ in \eqref{ineq:xy} by $i$, summing the resulting inequality from $i = 1$ to $k$, and using Lemma \ref{lem:avg}(a), convexity, and \eqref{def:spp-pt}, we have for every $(u,v) \in \R^n \times \R^m$,
\begin{align*}
    & h_1(\bar x_k) + f(\cdot,\bar y_k)^*(\bar s_k^x) - \inner{\bar s_k^x}{u} - h_1(u) + h_2(\bar y_k) + [-f(\bar x_k,\cdot)]^*(\bar s_k^y) - \inner{\bar s_k^y}{v} - h_2(v) \\
    \le& 16 \lam M^2 + \frac{1}{2\lam k}\|z_0 - w\|^2.
\end{align*}
Maximizing both sides of the above inequality over $(u,v) \in \dom h_1 \times \dom h_2$ yields
\begin{align*}
    &h_1(\bar x_k) + f(\cdot,\bar y_k)^*(\bar s_k^x) + h_1^*(-\bar s_k^x) + h_2(\bar y_k) +[-f(\bar x_k,\cdot)]^*(\bar s_k^y) + h_2^*(-\bar s_k^y) \\
    \le& 16 \lam M^2 + \frac{1}{2\lam k} \max\{\|z_0 - w\|^2: w\in \dom h_1 \times \dom h_2 \}.
\end{align*}
Finally, \eqref{ineq:Phi} follows from Lemma \ref{lem:avg}(b), (B3), and the definition of $\Phi(\cdot,\cdot)$ in \eqref{eq:spp-Phi}.
\end{proof}

\begin{theorem}\label{thm:cs-spp}
Given $(x_0,y_0,\bar \varepsilon) \in \dom h_1 \times \dom h_2 \times R_{++}$, letting $\lam = {\bar \varepsilon}/{32M^2}$, then the number of iterations of CS-SPP$(x_0,y_0,\lam)$ to find a $\bar \varepsilon$-saddle-point $(\bar x_k, \bar y_k)$ of \eqref{eq:spp} is at most $128 M^2 D^2/\bar \varepsilon^2$. 
\end{theorem}

\begin{proof}
    It follows from Proposition \ref{prop:spp-rate} and the choice of $\lam$ that
    \[
    \Phi(\bar x_k,\bar y_k) \le \frac{\bar \varepsilon}{2} + \frac{64 D^2}{\bar \varepsilon k}.
    \]
    Hence, the conclusion of the theorem follows immediately.
\end{proof}

\section{Implementation details of numerical experiments}\label{sec:detail}

This section presents Algorithm \ref{alg:Sj} for exactly solving \eqref{def:fz}, which gives rise to the exact computation of $\varphi(x)$ and $\psi(y)$ in \eqref{eq:varphi}. We first state a technical result that characterizes the optimal solution $\hat x$ to \eqref{def:fz}, from which Algorithm \ref{alg:Sj} follows as a direct consequence.

\begin{proposition}\label{prop:exact}
Let $z\in\R^n$ and scalar $\gamma \ge 0$ be given. 
Define
\begin{equation}\label{eq:Sj}
    S_j = \frac{1}{j}\left(\gamma+\sum_{i=1}^j z_{(i)}\right),
\end{equation}
where $(1), \ldots,(n)$ index $z$ in non-decreasing order $z_{(1)} \le \ldots \le z_{(n)}$.
Let $j^*$ be the first index such that $S_j \le S_{j+1}$, or $n$ if the condition is never satisfied.
Then, $\hat x \in \R^n$ defined as
\begin{equation}\label{def:xhat}
    \hat x_{(i)}=\begin{cases}
    \frac{1}{j^*},& i\leq j^*;\\
    0,& \text{otherwise}
    \end{cases}
\end{equation}
exactly solves \eqref{def:fz}.
\end{proposition}

\begin{proof}
Without loss of generality, we assume that $z$ has been sorted with non-decreasing entries, i.e., $z_1 \le \ldots \le z_n$.
It is easy to see that \eqref{def:fz} can be reformulated as
\[
\min_{x \in \R^n, t \ge 0} \left\{z^\top x + \gamma t: 0 \le x_i \le t, \, i=1, 2, \ldots, n, \, \sum_{i=1}^n x_i = 1\right\}.
\]
For fixed $t$, the optimal $x$ assigns as much mass as allowed (up to the capacity $t$) to the smallest coordinates of $z$. Hence, the optimal solution $\hat x$ has the form of
\[
\hat x_1 = \ldots = \hat x_j = \frac{1}{j}, \quad \hat x_{j+1} = \ldots = \hat x_n = 0,
\]
where $j \in\{1, \ldots, n\}$ satisfies $t=1/j$.
We thus note that the objective value in \eqref{def:fz} at $\hat x$ is
\[
    f_z(\hat x) \stackrel{\eqref{def:fz}}= \frac{1}{j}\left(\gamma+\sum_{i=1}^j z_i\right) \stackrel{\eqref{eq:Sj}}= S_j.
\]
Therefore, the problem reduces to $\min\{S_j: j = 1, \ldots, n\}$.
We observe from the definition of $S_j$ in \eqref{eq:Sj} that
\begin{equation}\label{eq:obs}
    S_{j+1} = \frac{j S_j + z_{j+1}}{j+1},
\end{equation}
and hence that
    \begin{equation}\label{eq:Sj-diff}
        S_j - S_{j+1} = \frac{S_j - z_{j+1}}{j+1}.
    \end{equation}
It thus follows from $S_j \le S_{j+1}$ that $S_j \le z_{j+1}$, which, together with \eqref{eq:obs} and the monotonicity of $\{z_j\}$, implies that
    \[
    S_{j+1} \stackrel{\eqref{eq:obs}} \le z_{j+1} \le z_{j+2}.
    \]
Hence, it follows from \eqref{eq:Sj-diff} with $j=j+1$ that $S_{j+1} \le S_{j+2}$. Therefore, $\{S_j\}$ is non-increasing for $j \le j^*$ while non-decreasing for $j \ge j^*$.
Finally, we conclude that $\hat x$ defined in \eqref{def:xhat} is an exact optimal solution to \eqref{def:fz}.
\end{proof}

The optimal solution $\hat x$ to \eqref{def:fz} may not be unique, as the problem $\min\{S_j : j = 1, \ldots, n\}$ can admit multiple minimizers, and each minimal index $j^*$ induces a corresponding $\hat x$ via \eqref{def:xhat}.
The following algorithm for exactly solving \eqref{def:fz} is natural from Proposition \ref{prop:exact}.

\begin{algorithm}[H]
\caption{Exact solving for \eqref{def:fz}}
\begin{algorithmic}\label{alg:Sj}
\REQUIRE  given $z \in \R^n$ and $\gamma \ge 0$
\STATE Sort $z$ in ascending order, compute $S_1$ as in \eqref{eq:Sj}, and set $j^*=n$;
\FOR{$j = 1, \ldots, n-1$}
\STATE Compute $S_{j+1}$ as in \eqref{eq:Sj}, if $S_j \le S_{j+1}$, then set $j^*=j$ and quit the loop;
\ENDFOR
\STATE Compute $\hat x$ as in \eqref{def:xhat} and set $f_z(\hat x) = S_{j^*}$.
\ENSURE $\hat x$ and $f_z(\hat x)$
\end{algorithmic} 
\end{algorithm}

\end{document}

%% file: def.tex
\usepackage{amsmath}
\usepackage{amsfonts}
\usepackage{latexsym}
\usepackage{amssymb}

\usepackage{xcolor}
\usepackage{float}
\usepackage{amsthm}
\usepackage{algorithm}

\usepackage{hyperref}
\hypersetup{hypertex=true, colorlinks=true, linkcolor=blue, anchorcolor=blue, citecolor=blue}
\usepackage{geometry}
\usepackage{enumerate}

\usepackage{cite}

\newtheorem{thm}{Theorem}[section]
\newtheorem{theorem}{Theorem}[section]

\newtheorem{lemma}[thm]{Lemma}

\newtheorem{proposition}[thm]{Proposition}

\newtheorem{remark}[thm]{Remark}

\newcommand{\beq}{\begin{equation}}
\newcommand{\eeq}{\end{equation}}
\newcommand{\beqa}{\begin{eqnarray}}
\newcommand{\eeqa}{\end{eqnarray}}
\newcommand{\beqas}{\begin{eqnarray*}}
\newcommand{\eeqas}{\end{eqnarray*}}
\newcommand{\bi}{\begin{itemize}}
\newcommand{\ei}{\end{itemize}}

\newcommand{\nn}{\nonumber}

\newcommand{\R}{\mathbb{R}}

\newcommand{\lam}{{\lambda}}

\newcommand{\inner}[2]{\langle #1,#2\rangle}

\newcommand{\argmin}{\mathrm{argmin}\,}
\newcommand{\argmax}{\mathrm{argmax}\,}

\newcommand{\dom}{\mathrm{dom}\,}

\newcommand{\Argmin}{\mathrm{Argmin}\,}

\newcommand{\bConv}[1]{\overline{\mbox{\rm Conv}}\,(\R^{#1})}

\newcommand{\tx}{\tilde x}
\newcommand{\ty}{\tilde y}

\newcommand{\infnorm}[1]{\left\|#1\right\|_{\infty}}